\documentclass[titlepage,reqno]{amsart}

\usepackage{amsmath,amsfonts,amssymb,amsthm}
\usepackage{cancel}
\usepackage{color}
\usepackage{xcolor}
\usepackage{graphics}

\usepackage[english]{babel}

\newcommand{\myspace}{\qquad\qquad\qquad}

\newcommand{\cA}{{\mathcal A}}

\newcommand{\cD}{{\mathcal D}}
\newcommand{\cF}{{\mathcal F}} 

\newcommand{\cH}{{\mathcal H}}
\newcommand{\cL}{{\mathcal L}}

\newcommand{\cT}{{\mathcal T}}

\definecolor{DarkOrange}{RGB}{255,140,000} 

\definecolor{DarkOrchid}{RGB}{153,050,204}

\definecolor{DarkRed}{rgb}{0.55,0.00,0.00}

\newtheorem{theorem}{Theorem}[section]
\newtheorem{lemma}[theorem]{Lemma}
\newtheorem{proposition}[theorem]{Proposition}
\newtheorem{remark}[theorem]{Remark}
\newtheorem{remarks}[theorem]{Remarks}

\newtheorem{hypotheses}[theorem]{Hypotheses}
\newtheorem{definition}[theorem]{Definition}
\newtheorem{corollary}[theorem]{Corollary}

\numberwithin{equation}{section}

\date{}

\title[On the regularity of solutions to the MGT equation]{On the regularity of solutions to the Moore-Gibson-Thompson equation:
a perspective via wave equations with memory}

\author{Francesca Bucci}
\address{
Francesca Bucci, Universit\`a degli Studi di Firenze,
{\sl Dipartimento di Matematica e Informatica}, {\rm Via S.~Marta 3, 50139 Firenze, ITALY}
}
\email{francesca.bucci(at)unifi.it}

\author{Luciano Pandolfi}
\address{Luciano Pandolfi, Politecnico di Torino, 
{\sl Dipartimento di Scienze Matematiche ``Giuseppe Luigi Lagrange''}, 
{\rm Corso Duca degli Abruzzi 24, 10129 Torino, ITALY}
}
\email{luciano.pandolfi(at)polito.it}

\begin{document}

 
\begin{abstract}
\noindent
We undertake a regularity analysis of the solutions to initial/boundary value problems
for the (third-order in time) Moore-Gibson-Thompson (MGT) equation.
The key to the present investigation is that the MGT equation falls within a large class of systems with memory, with affine term depending on a parameter. 
For this model equation a regularity theory is provided, which is of also independent interest;
it is shown in particular that the effect of boundary data that are square integrable
(in time and space) is the same displayed by wave equations.
Then, a general picture of the (interior) regularity of solutions corresponding to 
homogeneous boundary conditions is specifically derived for the MGT equation in various
functional settings. This confirms the gain of one unity in space regularity for the time
derivative of the unknown, a feature that sets the MGT equation apart from other PDE
models for wave propagation.
%
The adopted perspective and method of proof enables us to attain as well 
the (sharp) regularity of boundary traces. 

\end{abstract}

\maketitle


\section{Introduction}
The Jordan-Moore-Gibson-Thompson equation is a nonlinear Partial Differential Equation 
(PDE) model which describes the acoustic velocity potential in ultrasound wave propagation;
the use of the constitutive Cattaneo law for the heat flux, in place of the Fourier law,
accounts for its being of third order in time.
The quasilinear PDE is
\begin{equation}\label{Eq:quasilineare}
\tau \psi_{ttt} + \psi_{tt}-c^2\Delta \psi - b\Delta \psi_t=
\frac{\partial}{\partial t}\Big(\frac1{c^2}\frac{B}{2A}\psi^2_t+|\nabla \psi|^2\Big)
\end{equation}
in the unknown $\psi=\psi(t,x)$, that is the acoustic velocity potential (then 
$-\nabla \psi$ is the acoustic particle velocity), $A$ and $B$ being suitable constants; 
{\em cf.}~Moore \& Gibson \cite{moore-gibson_1960}, Thompson \cite{thompson_1972}, 
Jordan \cite{jordan_2009}. 
For a brief overview on nonlinear acoustics, along with a list of relevant references,
see the recent paper by Kaltenbacher \cite{kalt_2015}.
Aiming at the understanding of the nonlinear equation, a great deal of attention has
been recently devoted to its linearization---referred to in the literature as the 
Moore-Gibson-Thompson (MGT) equation---whose mathematical analysis is also of independent 
interest, posing already several questions and challenges. 

Let $\Omega\subset \mathbb{R}^n$ be a region with smooth ($C^2$) boundary 
$\Gamma:=\partial\Omega$.
(It is a natural conjecture that existence results for wave equations in non-smooth domains
({\em cf.}~\cite{Grisvard}) can be extended to wave equations with memory and to the MGT equation, by using the methods we present in this paper.
We consider the MGT equation
\begin{equation}\label{e:mgt}
\tau u_{ttt}+\alpha u_{tt} -c^2 \Delta u -b \Delta u_t =0 
\qquad \text{in $(0,T)\times\Omega$}
\end{equation}
in the unknown $u=u(t,x)$, $t\ge 0$, $x\in \Omega$, representing the acoustic velocity potential or alternatively, the acoustic pressure (see \cite{kalt-las-posp_2012}
for a discussion on this issue).
The coefficients $c$, $b$, $\alpha$ are constant and positive; they represent the speed and diffusivity of sound ($c$, $b$), and, respectively, a viscosity parameter ($\alpha$). 
For simplicity we set $\tau=1$ throughout the paper.
 
Equation \eqref{e:mgt} is supplemented with initial and boundary conditions:
\begin{align} 
& u(0,\cdot)=u_0\,,\; u_t(0,\cdot)=u_1\,,\; u_{tt}(0,\cdot)=u_2(x)\,, 
& \text{in $(0,T)\times\Omega$}
\label{e:IC}
\\[1mm]
& \cT u(t,\cdot) =g(t,\cdot)  
& \text{on $(0,T)\times\Gamma$};
\label{e:BC}
\end{align}
$\cT$ denotes here a boundary operator, which---for the sake of simplicity---associates to a function either the trace on $\Gamma$, or the outward normal derivative 
$\frac{\partial}{\partial \nu}\big|_\Gamma$ (it would be the {\em conormal} derivative,
in the case of a more general elliptic operator than the Laplacian).

The original studies of the MGT equation with homogenous (Dirichlet or Neumann) boundary data 
carried out in Kaltenbacher~{\em et al.} \cite{kalt-etal_2011} and Marchand {\em et al.} \cite{marchand-etal_2012}~establish appropriate functional settings for semigroup well-posedenss, as well as stability and spectral properties of the dynamics, depending on the
parameters values.
They obtain, in particular,
\begin{enumerate}
\item[i)]
that assuming $b>0$ the linear dynamics is governed by a strongly continuous {\em group} in
the function space $H^1_0(\Omega)\times H^1_0(\Omega)\times L^2(\Omega)$ (Dirichlet BC),
or ($H^1(\Omega)\times H^1(\Omega)\times L^2(\Omega)$ (Neumann BC); 
\item[ii)]
that in the case $b=0$ the associated initial/boundary value problems 
are ill-posed ({\em cf.}~Remark~\ref{r:role-of-b}); 
\item[iii)]
that the parameter $\gamma=\alpha - \tau c^2/b$ is a threshold of stability/instability: it must be positive, if the property of uniform stability is required.
\end{enumerate}
The critical role of $\gamma$ for a dissipative behaviour was recently pointed out 
also in Dell'Oro and Pata~\cite{delloro-pata_2016}, within the framework of viscoleasticity.
(We add that linear and true nonlinear variants of the MGT equation including an {\em additional} memory term have been the object of recent investigation; see \cite{las-jee_2017} and references therein.) 

\smallskip
Our interest lies in studying the regularity of the mapping
\begin{equation*} 
(u_0,u_1,u_2,g)\longmapsto u
\end{equation*}
that associates to initial and boundary data---taken in appropriate spaces---the corresponding solution $u=u(t,x)$ to the initial/boundary value problem (IBVP)
\eqref{e:mgt}-\eqref{e:IC}-\eqref{e:BC}.
(We note that the time and more often the space variable $x$ will generally not be esplicit,
unless when needed for the sake of clarity.)

As it will be shown in the paper, it will be the embedding of equation \eqref{e:mgt}
in a general class of integro-differential equations (depending on a parameter) 
to spark our method of proof for the regularity analysis of the associated initial/boundary value problems.
Indeed, the MGT equation is a special instance of the following wave equation with 
persistent memory, 
\begin{equation}\label{e:memory} 
u_{tt}-b \Delta u=-b\gamma \int_0^t N(t-s) \Delta u(s)\,ds + F(t)\xi\,,
\end{equation}
which displays an affine term depending on a suitable $\xi$, and that will be
supplemented with (initial and boundary) data
\begin{equation} \label{eq:dataDIe:memory}
u(0)=u_0\,,\ u_t(0)=u_1\,, \qquad \cT u=g\,.
\end{equation}
The assumptions on the real valued functions $N(t)$, $F(t)$ and on $\xi$ in \eqref{e:memory} 
are specified later; see Theorem~\ref{t:sample}.
As it will be apparent below, the parameter $\xi$ includes the component $u_2$ of
initial data $(u_0,u_1,u_2)$ for the MGT equation, while \eqref{e:memory}-\eqref{eq:dataDIe:memory} reduces to the MGT equation (with \eqref{e:IC}-\eqref{e:BC}) when
\begin{equation*}
N(t)=F(t)=e^{-\alpha t}\,,\qquad \xi=u_2-b\Delta u_0\,.
\end{equation*} 
 
\smallskip
The obtained regularity results will follow combining the (interior and trace) regularity
theory for wave equations with non-homogenous boundary data---the Neumann case being the most challenging (see \cite{las-trig_wave1}, and the optimal result of \cite{tataru_1998})---with the methods developed in~\cite{PandLIBRO} for equations with persistent memory.
In order to carry out a regularity analysis of the model equation with memory 
\eqref{e:memory} we shall use the trick of MacCamy \cite{maccamy_1977} and the theory of Volterra equations.

For equations with memory of the form \eqref{e:memory} the reader is referred, e.g., to \cite[Chapter~2]{PandLIBRO}; see also \cite[Chapter~5]{corduneanu}.
A novelty in the equation is brought about by the presence of the (vectorial) parameter $\xi$.
A classical reference on---and thorough treatment of---evolutionary integral equations is 
\cite{pruess_1993}.

\smallskip
It is important to emphasize at the very outset that the adopted perspective and approach
paves the way for establishing the (sharp) regularity of boundary traces for the MGT equation,  
as well as for the solutions to a rather general family of wave equations with memory, supplemented with Dirichlet or Neumann boundary conditions.
While being a topic of recognized current interest, the only result that appears available so far is the one obtained (via energy methods) in \cite{loreti-sforza_parma}, 
tailored for the case of Dirichlet boundary conditions.
Our alternative proof for the model equation with memory, depending on the parameter $\xi$
(and with the same BC) is given in Theorem~\ref{t:traces-memory}, which
brings about a boundary regularity result for the MGT equation, that is 
Corollary~\ref{c:traces-mgt}. 
(The study of the regularity of boundary traces for wave equations with memory in 
the case of Neumann boundary conditions is left to a separate, subsequent investigation.)


\subsection{Main results: synopsis} 
The outcome of the {\em interior} regularity analysis carried out in this paper is stated in 
Theorem~\ref{t:main_1}, pertaining to the general model equation with memory \eqref{e:memory},
and Theorem~\ref{t:main_2} for the MGT equation itself.
Beside being the former results instrumental in achieving the subsequent ones, they
are also of independent interest.
\\
Because the said results are presented by means of elaborate tables, aiming at rendering explicit the major achievements on the regularity of equations \eqref{e:memory} and 
\eqref{e:mgt}---the latter linked and complementing those in our key reference~\cite{kalt-etal_2011}---we highlight them in Theorem~\ref{t:sample} below.
Theorem~\ref{t:sample} includes as well a last statement on the regularity of {\em boundary} traces, an issue which is dealt with in Section~\ref{s:traces}; see Theorem~\ref{t:traces-memory} and Corollary~\ref{c:traces-mgt}.

\medskip
For the statement and understanding of all our findings, we need to introduce appropriate
functional spaces along with the related notation.
Let $A$ be the unbounded operator defined as follows:
\begin{equation} \label{definizOPERATORE-A}
A w:=(\Delta -I) w\,, \quad 
\cD(A) =\big\{w\in H^2(\Omega)\colon \cT w =0 \; 
\text{on $\Gamma$}\big\}\,;
\end{equation}
namely, $A$ is the (so called) {\em realization} of the differential operator $\Delta -I$ in 
$L^2(\Omega)$, with homogeneous boundary conditions (BC) defined by $\cT$, in the present work
of either Dirichlet or Neumann type; of course, the domain of $A$ depends on $\mathcal{T}$.
(We might take the realization of the laplacian in the case of Dirichlet BC; translating the
differential operator allows us to deal with both significant BC at once.)
We further note that 
$A$ is the infinitesimal generator of an exponentially stable analytic
semigroup and the fractional powers of $-A$ are well defined.
Thus, we are allowed to introduce the functional spaces $X_s$ definied as follows:
\begin{equation}\label{e:x_r} 
X_s=
\begin{cases} 
\cD((-A)^{s/2}) & \text{if $s \ge 0$}
\\[1mm]
[\cD((-A)^{s/2})]' & \text{if $s < 0$}\,,
\end{cases}
\end{equation}
endowed with the graph norm if $s\ge 0$, while the norm of a dual space is needed otherwise.
\\
Then, it is well known that $A\colon X_s \rightarrow  X_{s-2}$ is continuous, surjective and boundedly invertible.

\smallskip
The next theorem collects results obtained in Sections~\ref{s:main_1}, \ref{sect:RegulaMGT},
~\ref{s:traces}. 
\begin{theorem}[A compendium of main results] \label{t:sample}
The following assertions hold:
\begin{itemize}
\item[i)] \label{item1THTheorem:sample} 
{\rm 
(Interior regularity for the equation with memory~\eqref{e:memory}) with homogeneous BC).
}
Assume $N(t)\in H^2(0,T)$ and $F(t)\in L^2(0,T)$ for every $T>0$. 
Let $g\equiv 0$. 
If $u_0\in X_0$ and $u_1,\xi\in X_{-1}$,
then the solution $u$ to the IBVP problem \eqref{e:memory}-\eqref{eq:dataDIe:memory}
satisfies 
\begin{equation*}
u\in C([0,T];X_0)\cap C^1([0,T];X_{-1})\cap L^2(0,T;X_{-2})\,.
\end{equation*}

\item[ii)] \label{item2THTheorem:sample}
{\rm 
(Interior regularity for the MGT equation \eqref{e:mgt} with homogeneous BC).
}
If $g\equiv 0$ and $(u_0,u_1,u_2)\in X_1\times X_1\times X_0$,
then the solution $u$ to the IBVP problem \eqref{e:mgt}-\eqref{e:IC}-\eqref{e:BC}
satisfies 
\begin{equation} \label{e:boundary-tointerior}
u\in C([0,T];X_1)\cap C^1([0,T];X_1)\cap C^2([0,T];X_0)
\end{equation}
and the map $(u_0,u_1,u_2)\longmapsto u(t,x)$ is continuous in the specified spaces.

\item[iii)] \label{item3THTheorem:sample}
{\rm 
(Boundary-to-interior regularity for equations \eqref{e:memory} and \eqref{e:mgt},
with trivial initial data).
}
Assume $u_0=u_1=\xi=0$ ($u_0=u_1=u_2=0$, respectively), and $g\in L^2(0,T;L^2(\Gamma))$.
Then there exists $\alpha_0$ such that for the solutions $u$ to the IBVP problem 
\eqref{e:memory}-\eqref{eq:dataDIe:memory} (\eqref{e:mgt}-\eqref{e:IC}--\eqref{e:BC},
respectively) satisfy
\begin{equation} \label{e:boundary-to-interior}
u\in C([0,T];X_{\alpha_0})\cap C^1([0,T];X_{\alpha_0-1}\cap  L^2(0,T;X_{\alpha_0-2})\,;
\end{equation}
the value of $\alpha_0$ depends on the boundary operator $\mathcal{T}$ (and partly
on $\Omega$) and are specified in \eqref{e:alfazero} below.
\item[iv)] 
{\rm 
(Regularity of boundary traces for the MGT equation \eqref{e:mgt}).
}
Let $u=u(t,x)$ be a solution to the MGT equation \eqref{e:mgt} corresponding to
initial data $(u_0,u_1,u_2)$ and homogeneous boundary data.
Assume $(u_0,u_1,u_2)\in H^1_0(\Omega)\times L^2(\Omega)\times H^{-1}(\Omega)$,
along with the compatibility condition
\begin{equation*} 
u_2-\Delta u_0\in L^2(\Omega)\,.
\end{equation*}
Then, for every $T>0$ there exists $M=M_T$ such that
\begin{equation*}
\begin{split}
& \int_0^T\!\!\!\int_{\partial\Omega} \Big|\frac{\partial}{\partial\nu} u(x,t)\Big|^2 
d\sigma\,d t 
\le M\, \Big( \|u_0\|_{H^1_0(\Omega)}+\|u_1|_{L^2(\Omega)}^2 +
\\[1mm]
& \myspace \myspace \qquad
+\|u_2-\Delta u_0\|^2_{L^2(\Omega)}\Big)\,. 
\end{split}
\end{equation*}

\end{itemize}

\end{theorem}

\begin{remarks}
\begin{rm}
We see from the statements in i) and ii), respectively, that while the equation with memory
\eqref{e:memory} displays a somewhat expected regularity, namely, the same as most PDE models for wave propagation, the interior regularity of solutions to the MGT equation \eqref{e:mgt} under homogeneous boundary conditions improves. 
\\
Instead, the regularity result in iii)---that pertains to the case of trivial initial data ($u_0=u_1=u_2=0$) and non-homogeneous boundary data ($g\ne 0$)---is not improved by special
choices of the kernel $N(t)$, such as $N(t)=e^{-\alpha t}$.

It is worth mentioning that our analysis does not disclose that the dynamics of the MGT
equation \eqref{e:mgt} is governed by a {\em group}. 
Following the studies on well-poseness performed in \cite{kalt-etal_2011} and \cite{marchand-etal_2012}, the present study focuses on the regularity analysis of a general
class of PDE systems which are governed by {\em semigroups}---not necessarily groups---and
whose solutions generally display a lower regularity than the ones of equation \eqref{e:mgt}.
The higher interior regularity for the MGT equation is obtained when we particularize
the formulas, and exploiting the smoothness of the coefficients.
\end{rm}
\end{remarks}

\smallskip
We note that the values of $\alpha_0$ which occurr in \eqref{e:boundary-to-interior}---and which correspond to appropriate Sobolev exponents---are the ones
established in the case of (linear) hyperbolic equations with $L^2(\Sigma)$ boundary data 
(of either Dirichlet or Neumann type).
We record explicitly for the IBVP 
\begin{equation}\label{e:ibvp-wave}
\begin{cases}
u_{tt}=\Delta u - u+f & \text{in $(0,T)\times \Omega$}
\\[1mm]
u(0,\cdot)=u_0\,, \; u_t(0,\cdot)=u_1  & \text{in $\Omega$}
\\[1mm]
\cT u=g  & \text{on $(0,T)\times \Gamma$}
\end{cases}\,.
\end{equation}
a statement which embodies a complex of successive achivements; see the cited references.
\begin{theorem}[\cite{las-lions-trig}, \cite{las-trig_wave1}, \cite{tataru_1998}]  
\label{t:tataru}
Assume that $u_0, u_1=0$, $f= 0$, and $g\in L^2(\Sigma)$.
Then, the unique solution to the initial/boundary value problem \eqref{e:ibvp-wave}
satisfies
\begin{equation*}
(u,u_t)\in C([0,T];H^{\alpha_0}(\Omega) \times H^{\alpha_0-1}(\Omega))\,,
\end{equation*}
with 
\begin{equation}\label{e:alfazero}
\alpha_0 =
\begin{cases}
0 & \text{if $\cT$ is the Dirichlet trace operator}
\\[1mm]
\frac23 & \text{if $\cT$ is the Neumann trace operator and $\Omega$ is a 
smooth domain}
\\[1mm]
\frac34 & \text{if $\cT$ is the Neumann trace operator and $\Omega$ is a parallelepiped.}
\end{cases}
\end{equation} 
\end{theorem}
For a chronological overview with historical and technical remarks see, e.g.,
\cite[Notes on Chapter~8, p.~761]{las-trig-book}. 

We finally point out on the regularity of wave equations the recent progress of 
\cite{triggiani_2016}, dealing with the case of boundary data $g$ that are {\em not}
`smooth in space', e.g., $g\in L^2(0,T;H^{-1/2}(\Gamma))$.
In view of the approach taken in the present work, it is clear that the results obtained 
therein could be utilized as well in order to attain regularity results for equations
with memory and for the MGT equation under boundary data that are less
regular (than square integrable) in space.


\subsection{Orientation}
The plan of the paper is briefly outlined below.
For the reader's convenience and since these tools will be utilized throughout,
in Section~\ref{s:preliminaries-on-wave} we provide a minimal background and references on the
approach to linear wave equations 
via cosine operator theory.

In Section~\ref{sec:memory} we perform an analysis of the equation with
memory \eqref{e:memory} that encompasses the MGT equation.
An equivalent equation---in fact easier, since the convolution term therein does not involve
differential operators at all---is derived, which in turn results
in a Volterra equation of the second kind; see Proposition~\ref{p:volterra}.
This step will play a crucial role in the proof of our first regularity result, 
that is Theorem~\ref{t:main_1}, concerning the model equation with memory \eqref{e:memory}.
Section~\ref{s:main_1} is then almost entirely devoted to the proof of Theorem~\ref{t:main_1}.

In Section~\ref{sect:RegulaMGT} we return to the third order MGT equation and show how the (interior) regularity results specifically pertaining to the MGT equation, stated in 
Theorem~\ref{t:main_2}, follow as a consequence of Theorem~\ref{t:main_1}.
Finally, Section~\ref{s:traces} is devoted to the regularity of boundary traces;
see Theorem~\ref{t:traces-memory} and Corollary~\ref{c:traces-mgt}.
\\
A discussion and explanation of the introduced definition of solutions to the third order
(in time) equation under investigation is postponed to Appendix~\ref{a:def-solutions}.

 
\section{Preliminaries on wave equations} \label{s:preliminaries-on-wave}
Consider the initial/boundary value problem for a linear wave equation 
\eqref{e:ibvp-wave}. 
Since the methods of proof employed in the present work rely in a crucial way on the representation of solutions to wave equations by means of cosine operators, few lines
on this approach follow.
The reader is referred to \cite{belleni} and \cite{las-trig-cos}, were a first use of cosine operators is found in order to study equations with persistent memory and in the context of boundary control theory, respectively; see also the former contribution of \cite{sova_1966}.
We adopt here the notation of \cite{belleni} and \cite{fattorini}.

We shall use the operator $A$ in~\eqref{definizOPERATORE-A}, which is the realization of the translation $\Delta -I$ of the 
Laplacian in $L^2(\Omega)$, with suitable homogeneous boundary conditions, according to
a (boundary) operator $\cT$.
(In the Dirichlet case $A$ might be simply the realization of the Laplacian.)
As noted already, $A$ is boundely invertibile, i.e. $A^{-1}$ exists and it is bounded, 
in fact compact (even if $\cT$ represents the normal derivative on $\Gamma$).
It generates an exponentially stable analytic
semigroup and the fractional powers of $-A$ are well defined and we shall use the spaces $X_s$ in \eqref{e:x_r}   ($X_s$ has the graph
norm if $s\ge 0$, and the norm as a dual space otherwise).
We recall once more that $A$: $X_s \rightarrow  X_{s -2}$ is continuous, surjective  and boundedly invertible.

Next, we introduce the Green maps $G\in \cL(L^2(\Gamma),L^2(\Omega))$ defined
as follows:
\begin{equation}\label{e:green-map}
G\colon L^2(\Gamma)\ni \varphi\longmapsto G\varphi=:\psi \;
\Longleftrightarrow \;
\begin{cases}
\Delta \psi =\psi & \textrm{on $\Omega$}
\\[1mm]
\cT \psi=\varphi & \textrm{on $\Gamma$}\,.
\end{cases}\,;
\end{equation}
By elliptic theory, it is known that there exists an appropriate $s >0 $ such that 
$\text{im}\,G\subset X_s$ so that $A G\subset X_{s -2}$.
For instance, in the case of Dirichlet boundary conditions one has 
$\text{im}\,G= H^{1/2}(\Omega) \subset X_s$, with inclusion that holds true for any 
$s=1/2 -\sigma$, $0<\sigma<\frac{1}{2}$. 
 
Thus, it is known that the solution to the IBVP \eqref{e:ibvp-wave} is given by 
\begin{equation}\label{e:waves-explicit}
\begin{split}
u(t)& =R_+(t)u_0+\cA^{-1} R_-(t)u_1-\cA\int_0^t R_-(t-s)Gg(s)\,ds \,+ 
\\[1mm]
& \myspace +\cA^{-1} \int_0^t R_-(t-s)f(s)\,ds 
\end{split}
\end{equation}
where the operator $\cA$, and the families of operators $R_+(\cdot)$, $R_{-}(\cdot)$
are defined as follows:
\begin{equation}\label{eq:defiOperRpm}
\cA=i(-A)^{1/2}\,,\qquad 
R_+(t)=\frac{e^{\cA t}+e^{-\cA t}}{2}\,,
\qquad R_-(t)=\frac{e^{\cA t}-e^{-\cA t}}{2}\,,
\end{equation}
$R_+(t)$ being the strongly continuous {\em cosine} operator generated by $A$ in
$L^2(\Omega)$; see \cite{sova_1966}, \cite{fattorini}, \cite[Vol.~II]{las-trig-book}.

\begin{remark}
\begin{rm}
The previous definitions make sense because $\cA$ is the infinitesimal generator of a 
$C_0$-{\em group} of operators;
in particular, we have as well 
\begin{equation*}
X_s=\cD((i\cA)^s)=\cD(\cA^s) \qquad\qquad \text{if $s\ge 0$,}
\end{equation*}
and $\cA$ is bounded and boundedly invertible from $X_s$ to $X_{s-1}$ for every 
$s$.
\end{rm}
\end{remark} 

Computing the derivatives of \eqref{e:waves-explicit} we obtain the following equalities,
valid in $H^{-1}(\Omega)$ and $H^{-2}(\Omega)$, respectively:
\begin{equation}\label{e:waves-explicit-prime}
\begin{split}
u_t(t)& =\cA R_-(t)u_0+R_+(t)u_1-A\int_0^t R_+(t-s) Gg(s)\,ds\, +
\\[1mm]
& \myspace + \int_0^t R_+(t-s)f(s)\,ds\,,
\end{split}
\end{equation}
as well as
\begin{equation} \label{e:waves-explicit-second}
\begin{split}
u_{tt}(t)&=A R_+(t)u_0+\cA R_-(t)u_1  -AG g(t) -A\Big(\cA \int_0^t R_-(t-s) Gg(s)\,ds\Big)+
\\[1mm]
&  + f(t) +\cA\int_0^t R_-(t-s)f(s)\,ds =
\\[1mm]
&=Au(t)-AGg(t)+f(t)\,.
\end{split}
\end{equation}

\begin{remark}
\begin{rm}
If $f(\cdot)$ is of class $C^1([0,T])$ then it is possible to integrate by parts, 
like in
\begin{equation*}
\cA^{-1} \int_0^t R_-(t-s)f(s)\,ds=-A^{-1}\left [
f(t)-R_+(t)f(0)-\int_0^t R_+(t-s) f(s) d s
\right ]
\end{equation*}
which brings about a gain of one unity in space regularity.
The integration by parts is rigorously justified in~\cite[Lemma~5]{PandAMO}.
\end{rm}
\end{remark}

The explicit formula \eqref{e:waves-explicit}, along with \eqref{e:waves-explicit-prime}
and \eqref{e:waves-explicit-second} are among the keys for the following regularity
result.
The statement in iii) is by far the most challenging, as its proof is based on 
pseudo-differential methods and microlocal analysis.
\begin{theorem}\label{teo:propertyONDE}
Let $T>0$ be given, and $s\in \mathbb{R}$. 
The following statements hold true for the solutions to the initial/boundary value
problem~(\ref{e:ibvp-wave}).
\begin{enumerate}
\item[i)]
\label{teo:propertyONDE-1} 
Assume $g=0$, $f=0$. 
Then $(u_0,u_1)\longmapsto u(t)$ is continuous from $X_s\times X_{s-1}$ 
into $C([0,T],X_s)\cap C^1([0,T],X_{s-1})\cap C^2([0,T],X_{s-2})$.
\item[ii)]
\label{teo:propertyONDE-2} 
Assume $u_0=0$, $u_1=0$, $g=0$.
Then the map $f \longmapsto u(t)$ is continuous from $L^2(0,T;X_s)$ into 
$C([0,T],X_{s+1})\cap C^1([0,T),X_s)$ 
while $u_{tt}(t)-f(t)\in C([0,T],X_{s+1})$.  
\item[iii)]\label{teo:propertyONDE-3} 
Assume $u_0=0$, $u_1=0$, $f=0$. 
Then, there exists $\alpha_0\ge 0$---depending on $\cT$ and possibly on the geometry of 
$\Omega$---such that for every $g\in L^2((0,T)\times \Gamma)$ we have 
$u\in  C([0,T],X_{\alpha_0})\cap C^1([0,T];X_{\alpha_0-1})\cap C^2([0,T],X_{\alpha_0-2})$.
The mapping $g\longmapsto u$ is continuous in the indicated spaces. 
\end{enumerate}

\end{theorem}

\begin{remarks}
\begin{rm}
1. With reference to the assertion $iii)$ above, 
we remind the reader that the proper value of Sobolev exponent $\alpha_0$ are
given in \eqref{e:alfazero}.

\noindent
2. The properties stated in the previous Theorem justify \eqref{e:waves-explicit} as a formula for the solutions to the IBVP~\eqref{e:ibvp-wave}, since the following fact is easily checked: when $u_0, u_1\in \cD(\Omega)$ ($C^\infty(\Omega)$ functions with compact support), 
$f \in \cD((0,T)\times \Omega)$, $g \in \cD((0,T)\times \Gamma)$, then $u-Gg\in C([0,T],\cD(A))\cap C^1([0,T],\cD(A))\cap C^2([0,T];L^2(\Omega))$ and the following equality holds:
\begin{equation*}
u_{tt}(t)=A(u(t)-Gg(t))+f(t)\,,
\end{equation*}
along with $u(0)=u_0$, $u_t(0)=u_1$.
Thus, the boundary condition $\cT u=g$ is satisfied in the sense that $u(t)-Gg(t)\in \cD(A)$
for almost any $t$.

\end{rm}
\end{remarks}


\section{The MGT equation as an equation with memory}  \label{sec:memory}

We initially proceed formally.
Rewrite the first hand side of equation \eqref{e:mgt} as
\begin{equation}\label{eq:MGTPasso1}
\begin{split}
& u_{ttt}+\alpha u_{tt} -c^2 \Delta u -b \Delta u_t =
\\[1mm]
& \myspace
= \big(u_{tt}-b \Delta u\big)_t 
+\alpha \big(u_{tt}-b\Delta u\big) - c^2\Delta u +\alpha b \Delta u=
\\[1mm]
& \myspace
= \big(u_{tt}-b \Delta u\big)_t +\alpha \big(u_{tt}-b\Delta u\big)+b \gamma \Delta u =0
\end{split}
\end{equation}
where we recall that $\gamma = \alpha -c^2/b$.
Solving the equation 
\begin{equation*}
\big(u_{tt}-b \Delta u\big)_t =-\alpha \big(u_{tt}-\Delta u\big)-b \gamma \Delta u 
\end{equation*}
in the `unknown' $u_{tt}-b \Delta u$ gives the following integral equation in the unknown
(and in fact not yet defined as solution) $u$:
\begin{equation}\label{eq:MGTriscritta}
u_{tt}-b \Delta u= e^{-\alpha t}\xi-b\gamma \int_0^t e^{-\alpha (t-s)}\Delta u(s) \, ds\,,
\end{equation}
with $\xi= u_2-b\Delta u_0$.

Thus, in view of the obtained equation~\eqref{eq:MGTriscritta}, we consider the following 
(more general) model equation with persistent memory, depending on the parameter $\xi$:
\begin{equation} \label{eq:MEMORY} 
u_{tt}-b \Delta u= -b\gamma \int_0^t N(t-s)\Delta u(s) \, ds + F(t)\xi
\end{equation}
(already appeared---as \eqref{e:memory}---in the Introduction and recorded here for the 
reader's convenience; notice that both functions $N(t)$ and $F(t)$ equal $e^{-\alpha t}$
in the MGT equation). 


\begin{remark}\label{r:role-of-b}
\begin{rm}
If it happens that $\gamma=0$, then \eqref{eq:MGTriscritta} is nothing but a wave equation
with affine term 
$F(t)\xi$ and the regularity of the corresponding solutions
follows from Theorem~\ref{teo:propertyONDE}.
Thus, we explicitly assume $\gamma \ne 0$, and recall from the Introduction that $b>0$.
It is important to emphasize that in the case $b=0$ the problem is ill-posed, 
since the semigroup generation fails, as proved in~\cite[Theorem~1.1]{kalt-etal_2011};
instead, if $b<0$ then the PDE becomes a {\em nonlocal} elliptic equation of a kind
studied by Skubacevski\v{i} in~\cite{skubacevskii}.
\end{rm}
\end{remark}

The regularity analysis of equation \eqref{eq:MEMORY} is carried out under the assumptions listed below.
 
\begin{hypotheses}\label{a:kernel-and-affine}
i) The  coefficient $b$ is positive.
\hskip 1mm
ii) The memory kernel $N(t)$ and the function $F(t)$ 
are real valued;
$N(t)\in H^2(0,T)$ while $F(t)\in L^2(0,T)$ for every $T>0$.
\end{hypotheses}


%


 
\subsection{An equivalent Volterra integral equation}

A first step in our analysis is to show that we can get rid of the (second order) 
differential operator in the convolution term of \eqref{eq:MEMORY}.
To do so, let us preliminarly introduce the Volterra equation of the second kind
\begin{equation}\label{e:volterra-2-kind}
X(t)- \gamma \int_0^t N(t-s) X(s) \, ds = G(t)\,, \qquad t\in [0,T]\,.
\end{equation}
This equation has a unique solution $X(t)$ given by the following formula:
\begin{equation}\label{eq:soluFORMvolte}
X(t)=G(t)-\int_0^t R_0(t-s) G(s)\,ds\,.
\end{equation}
where $R_0(\cdot)$ is the (unique) solution to the integral equation 
\begin{equation}\label{e:resolvent-kernel}
R_0(t)-\gamma \int_0^t N(t-s) R_0(s)\,ds = -\gamma N(t)\,, \qquad t\in [0,T]\,.
\end{equation}
The function $t \longmapsto R_0(t)$ is the {\em resolvent kernel} of the Volterra equation.
An important observation is that $R_0\in H^2(0,T)$ since $N\in H^2(0,T)$ and 
\mbox{$R_0(0)=-\gamma N(0)$}. 
We then see (either from~\eqref{eq:soluFORMvolte} or from \eqref{e:volterra-2-kind}) that if 
$G(t)$ is continuous then $X(t)$ is continuous; if $G(t)$ is square integrable then 
$X(t)$ is square integrable.


\smallskip
We now perform several formal computations which will lead to a definition of the solutions
to equation \eqref{eq:MEMORY} (with appropriate initial and boundary data).
%
Rewrite the equation \eqref{eq:MEMORY} in the following different fashion, 
\begin{equation}\label{e:like-a-volterra}
\Delta u -\gamma \int_0^t N(t-s)\Delta u(s) \, ds =\frac{1}{b}\big(u_{tt}- F(t)\xi\big)\,,
\end{equation}
that is a Volterra integral equation of the second kind in the unknown $\Delta u$.
With reference to the general form \eqref{e:volterra-2-kind}, we have here 
\begin{equation*}
G(t)=\frac{1}{b}\big(u_{tt}- F(t)\xi\big)\,.
\end{equation*}
The formula \eqref{eq:soluFORMvolte} gives
\begin{equation*}
b\Delta u =u_{tt}- F(t)\xi- \int_0^t R_0(t-s)\big(u_{ss}(s)- F(s)\xi\big)\,ds\,,
\end{equation*}
where $R_0(\cdot)$ is the unique solution to the integral equation \eqref{e:resolvent-kernel},
as explained above. 
Since $R_0\in H^2(0,T)$ we are allowed to integrate by parts twice, thereby obtaining
\begin{equation*}
\begin{split}
b\Delta u & = u_{tt}- F(t)\xi
- \Big\{R_0(t-s)u_t(s)\big|_{s=0}^{s=t}-\int_0^t R_0'(t-s)u_s(s)\,ds\Big\} +
\\[1mm]
& \myspace + \int_0^t R_0(t-s)F(s)\xi\,ds=
\\[1mm]
& = u_{tt}- F(t)\xi-R_0(0)u_t(t) + R_0(t) u_1 - R_0'(0)u(t)-R_0'(t)u_0-
\\[1mm]
& \myspace - \int_0^t R_0''(t-s)u(s)\,ds +  \int_0^t R_0(t-s)F(s)\xi\,ds\,,
\end{split}
\end{equation*}
where the memory term does not contain differential operators.

\begin{remark}
\begin{rm}
The computations carried out so far---known as MacCamy's trick 
\cite{maccamy_1977}---are purely formal, since the solutions to the equation \eqref{eq:MEMORY}
have not yet been defined.
\end{rm} 
\end{remark}

The obtained equation is a wave equation perturbed by a persistent memory, namely,
\begin{equation*}
\begin{split}
u_{tt} & = b(\Delta-I) u + (R_0'(0)+b) u(t) + R_0(0)u_t(t) + \int_0^t R_0''(t-s)u(s)\,ds-
\\[1mm]
& \myspace -R_0'(t)u_0 - R_0(t) u_1 + \Big\{F(t)\xi -\int_0^t R_0(t-s)F(s)\xi\,ds\Big\}\,.
\end{split}
\end{equation*}
The introduction of the function
\begin{equation} \label{e:variable-v}
v(t) = e^{-\frac{1}{2}R_0(0)t} u(t)
\end{equation}
enables us to eliminate the term $R_0(0)u_t$, and to attain the following equation in the unknown $v$:
\begin{equation} \label{eq:MEMORY-for-v} 
v_{tt}=b (\Delta v-v)+ \int_0^t K(t-s)v(s)\,ds + \beta v(t)
+ \big(h_2(t)\xi+h_1(t)u_1+ h_0(t) u_0\big)\,,
\end{equation}
with the constant $\beta$ and the functions $K(\cdot)$, $h_i(\cdot)$, $i=0,1,2$ given
by the formulas below:
\begin{equation}\label{e:various-functions}
\begin{split}
& \text{$K(t) = e^{-\frac{1}{2}R_0(0)t} R_0''(t)$ is square integrable in $(0,T)$;}
\\
& \beta=b+\frac{1}{4}R_0^2(0)+R_0'(0)\,;
\\
& h_0(t)=e^{-\frac12 R_0(0) t} R_0'(t)\in H^1(0,T)\,;
\\
& h_1(t)=e^{-\frac12 R_0(0) t} R_0(t)\in H^2(0,T)\,;
\\
& \text{$h_2(t)= e^{-\frac12 R_0(0) t}\big(F(t)-\int_0^t R_0(t-s) F(s)\,ds\big)$
is square integrable.} 
\end{split}
\end{equation}
The above suggests the following Definition, which is rigorously justified in the Appendix.

\begin{definition} \label{d:def-solution}
Let $\cH$ be a Hilbert space. 
An $\cH$-valued function $t\longmapsto u(t)$ is a solution of equation \eqref{eq:MEMORY} 
supplemented with initial/boundary conditions \eqref{eq:dataDIe:memory} if the function 
$t\longmapsto v(t)$ defined in \eqref{e:variable-v} is an $\cH$-valued 
continuous function which solves the Volterra integral equation \eqref{eq:MEMORY-for-v},
with $\beta$, $K(\cdot)$, $h_i(\cdot)$, $i=0,1,2$ defined by \eqref{e:various-functions}.
\end{definition}

\begin{remark}
\begin{rm}
In the case $F(t)\equiv N(t)=e^{-\alpha t}$, then the above definition yields the definition
of solutions to the MGT equation \eqref{e:mgt}, 
with initial/boundary conditions \eqref{e:IC}-\eqref{e:BC}.

\end{rm} 
\end{remark}

On the basis of Definition~\ref{d:def-solution} we are led to study the regularity of solutions the following IBVP for the wave equation with memory \eqref{eq:MEMORY-for-v}, that is: 
\begin{equation}\label{ibvp-for-v}
\begin{cases}
v_{tt}=b (\Delta v-v)+ \int_0^t K(t-s)v(s)\,ds + \beta v(t)
+ \big(h_2(t)\xi+h_1(t)u_1+ h_0(t) u_0\big)
\\[1mm]
v(0,\cdot)=v_0\,, \; v_t(0,\cdot)=v_1
\\[1mm]
\cT v=e^{-\frac{1}{2}R_0(0)t} g\,,
\end{cases}
\end{equation}
where initial data are related to the ones of $u$ via the following relations:
\begin{equation} \label{eq:datiUdatiV}
v_0=u_0\,, \qquad v_1= u_1-\frac{1}{2}R_0(0)u_0\,.
\end{equation} 

The next Proposition connects the IBVP \eqref{ibvp-for-v} to a Volterra equation of the
second kind, with suitable kernel and affine term.

\begin{proposition} \label{p:volterra}
Any solution to the initial/boundary value problem \eqref{ibvp-for-v} solves the Volterra
equation 
\begin{equation} \label{eq:represent}
v(t)+\int_0^t L(t-s)v(s)\,ds=H(t)\,,
\end{equation}
where $L(\cdot)$ is a strongly continuous kernel defined by
\begin{equation} \label{eq:defiL}
L(t)v=-\frac{\beta}{\sqrt b} \cA^{-1}R_-(\sqrt b t) v
-\frac{1}{\sqrt b}\cA^{-1}\int_0^t R_-(\sqrt b(t-s)) K(s)v\,ds
\end{equation}
(and $K(\cdot)$ is defined explicitly in \eqref{e:various-functions}), while
the affine term $H(\cdot)$ is given by 
\begin{equation} \label{eq:defiH}
\begin{split}
H(t)&=\Big[R_+(\sqrt{b}t)-\frac{R_0(0)}{2\sqrt{b}}\cA^{-1}R_-(\sqrt{b}t)\Big]u_0
+\frac{1}{\sqrt{b}}\cA^{-1} R_-(\sqrt{b}t) u_1-
\\[1mm]
& \qquad -\sqrt{b} \cA\int_0^t R_-(\sqrt{b}(t-s)) G\,e^{-\frac{1}{2}R_0(0)s} g(s)\,ds+
\\[1mm]
& \qquad +\frac{1}{\sqrt{b}} \cA^{-1}\int_0^t R_-(\sqrt{b}(t-s))
\big[h_2 (s)\xi+ h_1(s)u_1+h_0(s)u_0\big]\, ds\,.
\end{split}
\end{equation}
We recall once more that $R_0(\cdot)$ is the (scalar) resolvent kernel defined---in terms of 
$N(\cdot)$---by the integral equation \eqref{e:resolvent-kernel}.
\end{proposition}

\begin{proof}
The proof is straightforward, in view of formula \eqref{e:waves-explicit} for the solution
to wave equations with initial and boundary data. 
We just recall here that the abstract operator $A$ is the realization of the differential operator $\Delta -I$ with boundary conditions driven by $\cT$, while $R_+(\sqrt{b}t)$, the {\em cosine} operator generated by $bA$, and $R_-(\sqrt{b}t)$ are defined in~(\ref{eq:defiOperRpm}).
We finally note that $H(\cdot) z\in C([0,T];X_\alpha)$, for every $z\in X_\alpha$.
\end{proof}


\section{Interior regularity for the equation \eqref{e:memory}}  \label{s:main_1}
In this Section we see how the regularity results pertaining to wave equations stated
in Theorem~\ref{teo:propertyONDE} can be suitably extended to the general equation with
memory of the form \eqref{eq:MEMORY}.
This will eventually imply the {\em stronger} regularity of solutions to the third order
MGT equation \eqref{e:mgt} (see the next Section). 

\smallskip
The key and starting point is the Volterra integral equation \eqref{eq:represent} in the
unknown $v$.
Its kernel $L(\cdot)$ is now operator valued and \emph{strongly continuous} from 
$[0,+\infty)$ to $\cL(X_\alpha)$ for every $\alpha$.
By using Theorem~\ref{teo:propertyONDE} we will derive the regularity properties of the
right hand side of \eqref{eq:represent}, that will be inherited by $v$ and then by the 
solutions to the wave equation with memory \eqref{eq:MEMORY}.
These properties will be expressed in terms of the boundary datum $g$, as well as of
initial data $u_0$, $u_1$ and $\xi$.

It is convenient to write explicitly the solution of a Volterra integral equation in a Hilbert space $\cH$. 
We introduce the notation $*$ for the convolution,
\begin{equation*}
L * h=\int_0^t L(t-s)h(s)\,ds=\int_0^t L(s)h(t-s)\,ds\,.
\end{equation*}
Here $L(t)$ is a strongly continuous function of time, with values in $\cL(\cH)$ and 
$h(t)$ is an integrable $\cH$-valued function. 
\\
Moreover, let $L^{(*n)}$ denote iterated convolutions, recursively defined by the following equalities
\begin{equation*}
L^{(*1)}=L\,,\quad  L^{(*(n+1))}*h=L*\Big( L^{(*n)}*h\Big) 
\end{equation*}
(for every integrable $\cH-$valued function $h$).
Then, the solution to the Volterra equation \eqref{eq:represent}---that is $v+L*v=H$, in
short---is 
\begin{equation*}
v=H+\sum _{k=2}^{\infty} L^{(*k)}*H\,.
\end{equation*}
Uniform convergence of the series is easily proved.
In the special case of our interest, $\cH=X_\alpha$ and $L(t)$ is given by \eqref{eq:defiL}.
The following result is well known.

\begin{lemma} \label{lemmaVOLTE}
Let $T>0$ and let the kernel $L(\cdot)$ be given by \eqref{eq:defiL}.
If $H\in C([0,T];X_\alpha)$, then the solution $v$ of the Volterra equation 
$v+L*v=H$ satisfies $v\in C([0,T];X_\alpha)$.  
\end{lemma}

We will repeatedly use Lemma~\ref{lemmaVOLTE} in order to pinpoint the regularity of the solutions
to the initial/boundary value problems associated with the equation \eqref{eq:MEMORY}.

\smallskip

\begin{theorem}[Regularity for equation \eqref{eq:MEMORY}]\label{t:main_1}
Consider Eq.~\eqref{eq:MEMORY} with initial data $(u_0,u_1)$ and boundary
data defined by \eqref{e:BC}.
Then, the regularity of the (linear) map $(u_0,u_1,\xi,g) \longmapsto (u,u_t,u_{tt})$
is detailed in Table~\ref{tableGENEresu}. 
\begin{table}[h]
\begin{center}

\begin{tabular}{|c|c|c|c|c|}
\hline 
&&&&\\
$u_0$ & $u_1$ & $\xi$ & $g $ & $u=u(t,x)$ solution of \eqref{eq:MEMORY}\\
&&&&
\\

 &&&&
\\
 \hline
&&&&
\\
$X_\lambda$ &    $0$ &    $0$ &    $0$ & 
$C([0,T];X_\lambda)\cap C^1([0,T];X_{\lambda-1})\cap C^2([0,T];X_{\lambda-2})$
\\
&&&&\\
\hline
&&&&\\
$0$ &   $X_\mu$ & $0$ &   $0$ & $C([0,T];X_{\mu+1})\cap C^1([0,T];X_\mu)\cap C^2([0,T]; X_{\mu-1})$ 
\\
&&&&\\
\hline
&&&&\\
$0$ & $0$ & $X_\nu$ & $0$ &$\left\{\begin{array}{l}
\mbox{if $F(t)\in L^2(0,T)$   then:}\\
  C([0,T];X_{\nu+1})\cap C^1([0,T]; X_{\nu })
\cap H^2([0,T]; X_{\nu-1})\,; \\
\mbox{if $F(t)\in H^1(0,T)$ then:}\\
   C([0,T];X_{\nu+2})\cap C^1([0,T]; X_{\nu+1})
\cap H^2([0,T]; X_\nu) 
\end{array}\right.$
\\
&&&&\\
  \hline 
  &&&&\\
 $0$ &    $0$ &$0$ & $ L^2(\Sigma) $ & 
$C([0,T];X_{\alpha_0})\cap C^1([0,T];X_{\alpha_0-1 })\cap H^2([0,T]; X_{\alpha_0-2})$  
\\
&&&&\\
  \hline
\end{tabular}
\caption{Regularity of solutions to the equation \eqref{eq:MEMORY}. The transformations are continuous between the indicated spaces.} 
\label{tableGENEresu}
\end{center}
\end{table}

\end{theorem}
\begin{proof}
The proof of the several statements contained in Table~\ref{tableGENEresu}
is structured in few major steps.

\smallskip
\paragraph{\bf 0. Premise and outline.}
Consider the Volterra equation \eqref{eq:represent}, and note that the functions 
$v_t(t)$ and $v_{tt}(t)$ solve the same Volterra integral equation of the second
kind in a Hilbert space, yet with different affine terms, $H_1(\cdot)$ and $H_2(\cdot)$ say,
respectively, which will be computed in the next step.
In view of Lemma~\ref{lemmaVOLTE}, the (time and space) regularity of these affine terms---depending on $u_0$, $u_1$, $\xi$ and $g$---will naturally bring about the (time and space) regularity for the triple $(v,v_t,v_{tt})$.
\\
To do so we will set to zero all data but one.
%
Finally, the derived regularity properties will be inherited by the triple 
$(u,u_t,u_{tt})$ pertaining to the original equation with persistent memory 
\eqref{eq:MEMORY}, still depending on $u_0$, $u_1$, $\xi$ and $g$.

\smallskip
\paragraph{\bf 1. The affine terms of Volterra equations.}
We rewrite \eqref{eq:represent} in the form
\begin{equation*}
v(t)+\int_0^t L(s)v(t-s) ds= H(t)
\end{equation*}
and compute the derivatives of both the sides. 
Inserting the expressions \eqref{eq:defiL} and \eqref{eq:defiH} of $L(t)$ and $H(t)$, and
replacing initial data $v_0$ and $v_1$ with their respective expressions in terms of $u_0$
and $u_1$ (see \eqref{eq:datiUdatiV}), we obtain the following Volterra integral equations
in the unknowns $v_t(t)$ and $v_{tt}(t)$:
\begin{align}
v_t(t)+\int_0^t L(t-s)v_s(s)\,ds= H_1(t)\,,
\label{eq:deriPRIMA}
\\[1mm]
v_{tt}(t)+\int_0^t L(t-s)v_{ss}(s)\,ds = H_2(t)
\label{eq:deriSEconda}
\end{align}
where  
%
\begin{align}
\nonumber & H_1(t):= L(t)v_0+H_t(t)=
\\
&=\Big[\frac{\beta}{\sqrt{b}}\cA^{-1} R_-(\sqrt{b} t)u_0 +
  \frac{1}{\sqrt b}\int_0^t R_-(\sqrt{b}(t-s))\cA^{-1}K(s)u_0\,ds\Big] +H_t(t)\,,
\label{e:y_1}
\\[2mm]
& H _2(t):= \Big[\frac{\beta}{\sqrt b}\cA^{-1} R_-(\sqrt{b}t)u_1
+\frac{1}{\sqrt b} \cA^{-1}\int_0^t R_-(\sqrt{b}(t-s))K(s)u_1\,ds \Big] -
\nonumber \\[1mm]
& \qquad\qquad-\frac{R_0(0)}{2\sqrt b}\Big[\beta\cA^{-1}R_-(\sqrt{b}t)u_0
+\cA^{-1}\int_0^t R_-(\sqrt{b}(t-s))K(s)u_0 \,ds \Big]+
\nonumber \\[1mm]
& \myspace +\beta R_+(\sqrt{b}t)u_0+\int_0^t R_+(\sqrt{b}(t-s)) K(s)u_0\,ds +H_{tt}(t)\,,
\nonumber
\end{align}
while the explicit expression \eqref{eq:defiH} of $H(t)$ is recorded here
for the reader's convenience:
\begin{equation*}
\begin{split}
H(t)&=\Big[R_+(\sqrt{b}t)-\frac{R_0(0)}{2\sqrt{b}}\cA^{-1}R_-(\sqrt{b}t)\Big]u_0
+\frac{1}{\sqrt{b}}\cA^{-1} R_-(\sqrt{b}t) u_1-
\\[1mm]
& \qquad -\sqrt{b} \cA\int_0^t R_-(\sqrt{b}(t-s)) G\,e^{-\frac{1}{2}R_0(0)s} g(s)\,ds+
\\[1mm]
& \qquad +\frac{1}{\sqrt{b}} \cA^{-1}\int_0^t R_-(\sqrt{b}(t-s))
\big[h_2 (s)\xi+ h_1(s)u_1+h_0(s)u_0\big]\, ds\,.
\end{split}
\end{equation*}
As it will appear clear immediately below, we neglected to write explicitly the derivatives
of $H(t)$, just because their regularity is easily deduced invoking once more 
Theorem~\ref{teo:propertyONDE}.

\smallskip
\paragraph{\bf 2a. Effects of boundary data action.}
With $u_0,\,u_1,\,\xi\equiv 0$, $g\in L^2(\Sigma)$, the affine term $H(t)$ 
in \eqref{eq:represent} (recorded above) reduces to 
\begin{equation}
H(t)=-\sqrt{b}\cA\int_0^t R_-(\sqrt{b}(t-s))Gg(s)\, ds\,.
\end{equation}
Therefore we know from assertion iii) \eqref{teo:propertyONDE-3} of 
Theorem~\ref{teo:propertyONDE} that 
\begin{equation*}
(H,H_t,H_{tt})\in C([0,T];X_{\alpha_0}\times X_{\alpha_0-1}\times X_{\alpha_0-2})\,.
\end{equation*}
Thus, Lemma~\ref{lemmaVOLTE} shows that the solutions of the Volterra 
equation \eqref{eq:represent} as well as those pertaining to the former
equation with memory \eqref{eq:MEMORY}
belong to
\begin{equation*} 
C([0,T];X_{\alpha_0})\cap C^1 ([0,T];X_{\alpha_0-1})\cap C^2([0,T];X_{\alpha_0-2})\,.
\end{equation*}
 
\smallskip
\paragraph{\bf 2b. Effects of the initial datum $u_0$.}
Assume $u_1, \,\xi=0$, $g=0$, and $u_0\in X_\lambda$. 
The affine term of the equation \eqref{eq:represent} in the unknown $v$ becomes
\begin{equation*}
H(t) = \Big[R_+(\sqrt{b}t)-\frac{R_0(0)}{2\sqrt{b}}\cA^{-1}R_-(\sqrt{b}t)\Big]u_0
+\frac{1}{\sqrt{b}} \cA^{-1}\int_0^t R_-(\sqrt{b}(t-s)) h_0(s)u_0\, ds\,,
\end{equation*}
so that readily
\begin{equation*}
H\in C([0,T];X_\lambda)\cap C^1([0,T];X_{\lambda-1})\cap C^2([0,T];X_{\lambda-2})\,,
\end{equation*}
which immediately implies $v(t)\in C([0,T];X_\lambda)$.
Recall now   the term $H_1$ in \eqref{e:y_1} and notice that its regularity is determined
by the regularity of $H_t$.  
Then, $H_1$---as well as $v_t$, in view of Lemma~\ref{lemmaVOLTE}---belongs to 
$C^1([0,T];X_{\lambda-1})$, while $H_2$ and then $v_{tt}(t)$ belong to 
$C^1([0,T];X_{\lambda-2})$, which establishes the first row of Table~\ref{tableGENEresu}.

\smallskip
\paragraph{\bf 2c. Effect of the initial datum $u_1$.} 
Assume $u_0,\ \xi=0$, and $g=0$ while $u_1\in X_\mu$.
In this case 
\begin{equation*}
H(t) = \frac{1}{\sqrt{b}}\cA^{-1} R_-(\sqrt{b}t) u_1
+\frac{1}{\sqrt{b}} \cA^{-1}\int_0^t R_-(\sqrt{b}(t-s)) h_1(s)u_1\, ds\,,
\end{equation*}
so that we have a slight regularization
\begin{equation*}
(H,H_t,H_{tt})\in C([0,T];X_{\mu +1}\times X_\mu\times X_{\mu-1});
\end{equation*}
the transformation $u_1\longmapsto H$ is continuous in the indicated spaces 
({\em cf.}~assertion {\em iii)} 
of Theorem~\ref{teo:propertyONDE}).
\\
The obtained regularity for $H$ and its derivatives holds for $H_i$, $i=1,2$, and then
is inherited by the solution $v(t)$: namely, 
\begin{equation*}
v\in C([0,T];X_{\mu+1})\cap C^1([0,T];X_\mu)\cap C^2([0,T]; X_{\mu-1})\,;
\end{equation*}
in turn, the same is valid for $u$, thereby confirming the second row of 
Table~\ref{tableGENEresu}.

\smallskip
\paragraph{\bf 2d. Effect of the parameter $\xi$.}  
We finally discuss the dependence on $\xi$. 
Assume $u_0, u_1=0$, and $g=0$ and $\xi\in X_\nu$. In this case
\begin{equation*}
H(t)=\frac{1}{\sqrt{b}} \cA^{-1}\int_0^t R_-(\sqrt{b}(t-s)) h_2 (s)\xi \,ds
\end{equation*}
and, from \eqref{e:various-functions}, $h_2(t)\in L^2(0,T)$, just like $F(t)$.

We invoke once more item {\em ii)} 
of Theorem~\ref{teo:propertyONDE}, and ascertain again a slightly regularizing property: the transformation $\xi\longmapsto v$ is continuous from $X_\nu$ to 
$C([0,T];X_{\nu+1})\cap C^1([0,T];X_\nu)\cap H^2([0,T]; X_{\nu-1})$ 
(while if in addition $F(t)$---and consequently, $h_2(t)$---is continuous, then $v\in C^2([0,T];X_{\nu-1})$).

In the case $F\in H ^2(0,T)$ (as the case of the MGT equation) we have a stronger regularization, since we can integrate by parts as follows:
\begin{equation}
\begin{split}
H(t) &=- \frac{1}{b}\,A^{-1}\int_0^t \frac{d}{ds }R_+(\sqrt{b}(t-s)) h_2 (s)\xi \, ds=
\\[1mm]
& \qquad=- \frac{1}{b}\,A^{-1}\Big[\big(h_2(t)-R_+(\sqrt{b}t) h_2(0)\big)\xi  
-\int_0^t  R_+(\sqrt{b}(t-s)) h_2'(s)\xi \,ds\Big]\,;
\end{split}
\end{equation}
a rigorous justification is found, e.g., in \cite[Lemma~5]{PandAMO}.

For a better understanding, we compute explicitly
\begin{equation*}
\begin{split}
H_t(t) &=-\frac{1}{b}\, A^{-1}\Big[\cancel{h_2'(t)\xi}-\sqrt{b}\cA R_-(\sqrt{b}t)\xi 
-\cancel{h_2'(t)\xi}-
\\[1mm]
& \myspace -\sqrt{b}\cA \int_0^t R_-(\sqrt{b}(t-s)) h_2' (s)\xi \,ds\Big]=
\\[1mm]
&= \frac{1}{b} \cA^{-1} R_-(\sqrt{b}t)\xi + \frac{1}{b} \cA^{-1}\int_0^t R_-(\sqrt{b}(t-s)) h_2'(s)\xi \,ds\in X_{\nu+1}\,,
\end{split}
\end{equation*}
and 
\begin{equation*}
H_{tt}(t) =R_+(\sqrt{b}t)\xi + \int_0^t R_+(\sqrt{b}(t-s)) h_2'(s)\xi \,ds\in X_{\nu}\,,
\end{equation*}
that complete the membership $H\in X_{\nu+2}$.

It is important to note that the space regularity increases of one unity and we get the 
result in the third row of Table~\ref{tableGENEresu}, where $H^2$ is replaced with $C^2$ 
if furthermore $F\in C^2(0,T)$, that is the case of the MGT equation.
\end{proof}


\begin{remark}
\begin{rm} 
The noticeable outcome of the obtained regularity result is that $u_1$ and $\xi$ are regularized by one and, respectively, two unities. 
Hence, when $g=0$, $u_0=0$ while $u_1$ and $\xi$ belong to $X_r$, then 
$(u(t),u_t(t),u_{tt}(t))$ evolves in $X_{r+1}\times X_r\times X_{r-1}$.
\end{rm}
\end{remark}

From Table~\ref{tableGENEresu} of Theorem~\ref{t:main_1} we deduce, in particular, 
the following regularity result. 

\begin{corollary}\label{c:example}
Consider equation \eqref{eq:MEMORY} with initial data $(u_0,u_1)$, and homogeneous
boundary data, namely, $g\equiv 0$ in \eqref{e:BC}.
Then, if $F(\cdot)\in C^1$ and if $(u_0,u_1,\xi)\in X_r\times X_{r-1} \times X_{r-2}$,
then the corresponding weak solution satisfies
\begin{equation*}
(u,u_t,u_{tt})\in C([0,T];X_r\times X_{r-1}\times X_{r-2})\,.
\end{equation*}

\end{corollary}
 

\section{Interior regularity for the MGT equation}\label{sect:RegulaMGT}

In this Section we utilize the analysis performed for the general class of equations
with memory \eqref{eq:MEMORY}, in order to derive a result pertaining
to the MGT equation, that is Theorem~\ref{t:main_2} below.
This Theorem establishes, in particular, the statements of Theorem~\ref{t:sample} detailing
the regularity from the boundary to the interior for the MGT equation (i.e. item {\em iii)}),
as well as the one pertaining to the interior regularity, under homogeneous boundary data
(i.e. item {\em ii)}). 
The latter result is consistent with the analyis formerly carried out in
\cite{kalt-etal_2011}, that brought about semigroup well-posedness
of the MGT equation in the space $\cD(A^{1/2})\times \cD(A^{1/2}) \times L^2(\Omega)$,
$A$ being the proper realization of the Laplacian in $L^2(\Omega)$; 
see \cite[Theorem~1.2]{kalt-etal_2011}.
The peculiar regularity of the MGT equation is here (re)confirmed in a wealth of functional
settings.

Recall that for the special case of the MGT equation we have in particular
\begin{equation*}
N(t)=F(t) = e^{-\alpha t}\,, \qquad \xi = u_2-b\Delta u_0 
\end{equation*}
in \eqref{eq:MEMORY}.
The meaning given to solutions is still the one   in Definition~\ref{d:def-solution}.

We restart from the integral equation which defines the resolvent 
associated with the convolution kernel $-\gamma N(t)$ of \eqref{eq:MEMORY},
that is equation \eqref{e:resolvent-kernel} and that in the present case---with
$N(t)= e^{-\alpha t}$---reads as 
\begin{equation}\label{e:resolvent-eq-mgt}
R_0(t)-\gamma \int_0^t e^{-(t-s)}\, R_0(s)\, ds= - \gamma e^{-\alpha t}\,. 
\end{equation}
It is then easily verified that the solution to \eqref{e:resolvent-eq-mgt} is given by
\begin{equation}\label{e:kernel-mgt}
R_0(t)= -\gamma e^{(\gamma-\alpha)t}=-\gamma e^{-\frac{c^2}{b}t}\,,  
\end{equation}
which gives $R_0(0)= -\gamma$ and hence
\begin{equation*}
\qquad v(t)= e^{\frac{\gamma}2 t}u(t)
\end{equation*} 
for $v$ defined in \eqref{e:variable-v}).

\smallskip
In view of Definition~\ref{d:def-solution}, and taking into account the actual expression
of $R_0(t)$---depending on $N(t)$ and $F(t)$---in \eqref{e:kernel-mgt}, the following instance
of Definition~\ref{d:def-solution} comes into the picture.

\begin{definition}[Instance of Definition~\ref{d:def-solution}]
A function $u=u(t,x)$ is a solution of the IBVP \eqref{e:mgt}-\eqref{e:IC}-\eqref{e:BC}
for the Moore-Gibson-Thompson equationif and only if the function 
$v(t,x) = e^{\frac{\gamma}{2}t}u(t,x)$ solves
\begin{equation*} 
\begin{split}
v_{tt} &=b (\Delta v-v)+ \int_0^t K(t-s)v(s)\,ds + \beta v(t)
\\[1mm]
& \qquad \qquad + h_2(t)\big(u_2-b\Delta u_0\big)+h_1(t)u_1+ h_0(t) u_0\,,
\end{split}
\end{equation*}
where 
\begin{equation*}
\begin{split}
& K(t) = -\gamma (\gamma-\alpha)^2 e^{(\frac{3}{2}\gamma-\alpha)t}\,,
\qquad \beta=b-\gamma\Big(\frac{3}{4}\gamma -\alpha\Big)\,,
\\
& h_0(t)=-\gamma (\gamma-\alpha) e^{(\frac{3}{2}\gamma -\alpha)t}\,,
\quad 
h_1(t)=-\gamma e^{(\frac{3}{2}\gamma -\alpha)t}\,,
\quad
h_2(t)= e^{(\frac{3}{2}\gamma-\alpha)t}\,. 
\end{split}
\end{equation*}

\end{definition}

\medskip
Thus, on the basis of Theorem~\ref{t:main_1}, we develop a picture of the interior
regularity for the MGT equation.

\begin{theorem}[Regularity for the MGT equation]\label{t:main_2} 
The regularity of the map $(u_0,u_1,u_2,g)\longmapsto u$ that associates to initial and 
boundary data the solution $u=u(t,x)$ 
to the Moore-Gibson-Thompson equation \eqref{e:mgt} is described by the Table~\ref{tableMGTresu}.
\begin{table}[h]
\begin{center}
\begin{tabular}{|c|c|c|c|c|}
\hline
&&&&\\
$u_0$ & $u_1$ & $u_2$ & $g $ & $u=u(t,x)$ solution of the MGT equation\\
&&&&
\\
 &&&&
\\
\hline
&&&&
\\
$X_\lambda$ &    $0$ &    $0$ &    $0$ & 
$C([0,T];X_\lambda)\cap C^1([0,T];X_{\lambda })\cap C^2([0,T];X_{\lambda-1})$
\\
&&&&\\
\hline
&&&&\\
$0$ &   $X_\mu$ & $0$ &   $0$ & $C([0,T];X_{\mu+1})\cap C^1([0,T];X_\mu)\cap C^2([0,T]; X_{\mu-1})$ 
\\
&&&&\\
\hline
&&&&\\
$0$ & $0$ & $X_\nu$ & $0$ &  $C([0,T];X_{\nu+2})\cap C^1([0,T]; X_{\nu+1})\cap C^2([0,T]; X_{\nu })$
\\
&&&&\\
  \hline 
  &&&&\\
 $0$ &    $0$ &$0$ & $ L^2(\Sigma) $ & 
$C([0,T];X_{\alpha_0})\cap C^1([0,T];X_{\alpha_0-1 })\cap H^2([0,T]; X_{\alpha_0-2})$  
\\
&&&&\\
  \hline
\end{tabular}
\caption{Regularity of solutions to the equation \eqref{e:mgt}. The transformations are continuous between the indicated spaces.} 
\label{tableMGTresu}
\end{center}
\end{table}

\end{theorem} 
\begin{proof}
Along the lines of the first steps of the proof of Theorem~\ref{t:main_1}, we return to
the Volterra equation \eqref{eq:represent} and note that the affine term $H(t)$ in
\eqref{eq:defiH} must be rewritten taking into account that in the present case
we have $\xi=u_2-b\Delta u_0$.
Let us focus on the last summand of $H(t)$, that is 
\begin{equation*}
\frac{1}{\sqrt{b}} \cA^{-1}\int_0^t R_-(\sqrt{b}(t-s))
\big[h_2 (s)\big(u_2-b\Delta u_0\big)+ h_1(s)u_1+h_0(s)u_0\big]\, ds\,,
\end{equation*}
and more specifically on the term
\begin{equation*}
T(t)=\frac{1}{\sqrt{b}} \cA^{-1}\int_0^t R_-(\sqrt{b}(t-s))h_2 (s)[u_2-b\Delta u_0]\,ds\,.
\end{equation*}
We rewrite 
\begin{equation}\label{e:t}
\begin{split}
T(t)&=\underbrace{\frac{1}{\sqrt{b}} \cA^{-1}\int_0^t R_-(\sqrt{b}(t-s))h_2 (s)u_2\,ds}_{T_1(t)}-
\\[1mm]
& \qquad \underbrace{-\sqrt{b} \cA^{-1}\int_0^t R_-(\sqrt{b}(t-s))h_2 (s)\Delta u_0\,ds}_{T_2(t)}
\end{split}
\end{equation}
and compute 
\begin{equation}\label{e:t_2}
\begin{split}
T_2(t)&=-\sqrt{b} \cA^{-1}\int_0^t R_-(\sqrt{b}(t-s))h_2 (s)\big(\Delta u_0-u_0+u_0\big)\,ds=
\\[1mm]
&=\underbrace{-\sqrt{b} \cA^{-1}\int_0^t R_-(\sqrt{b}(t-s))h_2 (s)
Au_0\,ds}_{T_{21}(t)}-
\\[1mm]
& \myspace
\underbrace{-\sqrt{b} \cA^{-1}\int_0^t R_-(\sqrt{b}(t-s))h_2 (s)u_0\,ds}_{T_{22}(t)}\,.
\end{split}
\end{equation}
Assuming $u_0\in X_\lambda$, then $Au_0\in X_{\lambda-2}$; moreover, recall that
$A=\cA^2$, and the relation between the operators $R_-(\cdot)$ and $R_+(\cdot)$.
Then proceed with the computations, integrating by parts to get  
\begin{equation}\label{e:t_21}
\begin{split}
T_{21}(t)& =-\sqrt{b} \cA^{-1}\int_0^t R_-(\sqrt{b}(t-s))h_2 (s)Au_0\,ds
\\[1mm]
& = -\sqrt{b} \cA\int_0^t R_-(\sqrt{b}(t-s))h_2 (s)u_0\,ds=
\\[1mm]
& =\int_0^t \frac{d}{ds} \big\{R_+(\sqrt{b}(t-s)) h_2 (s)u_0\Big\}\,ds
- \int_0^t R_+(\sqrt{b}(t-s))h_2'(s)u_0\,ds=
\\[1mm]
& =h_2(t)u_0-R_+(\sqrt{b}(t))h_2(0)u_0
-\int_0^t R_+(\sqrt{b}(t-s))h_2'(s)u_0\,ds\,.
\end{split}
\end{equation} 

Combine \eqref{e:t_21} with \eqref{e:t_2} and \eqref{e:t}, insert the resulting expression of 
$T(t)$ in $H(t)$, to obtain
\begin{equation*}
\begin{split}
H(t)& =\cancel{R_+(\sqrt{b}t)u_0}-\frac{R_0(0)}{2\sqrt{b}}\cA^{-1} R_-(\sqrt{b}t)u_0
+\frac{1}{\sqrt{b}}\cA^{-1} R_-(\sqrt{b}t) u_1-
\\[1mm]
& \qquad -\sqrt{b} \cA\int_0^t R_-(\sqrt{b}(t-s)) G\,e^{-\frac{1}{2}R_0(0)s} g(s)\,ds+
\\[1mm]
& \qquad
+\frac{1}{\sqrt{b}} \cA^{-1}\int_0^t R_-(\sqrt{b}(t-s))h_2 (s)u_2\, ds+
\\[1mm]
& \qquad
+ h_2(t) u_0\cancel{- R_+(\sqrt{b}t)u_0}
- \int_0^t R_+(\sqrt{b}(t-s))h_2'(s)u_0\, ds-
\\[1mm]
& \qquad
- \sqrt{b} \cA^{-1}\int_0^t R_-(\sqrt{b}(t-s))h_2 (s)u_0\,ds+
\\[1mm]
& \qquad
+ \frac{1}{\sqrt{b}} \cA^{-1}\int_0^t R_-(\sqrt{b}(t-s))\big[h_1(s)u_1+h_0(s)u_0\big]\,ds\,,
\end{split}
\end{equation*}
where the term $R_+(\sqrt{b}t)u_0$ appears twice with opposite signs, and hence cancel. 

Rearranging the summands and replacing $t-s$ with $s$ in the integrals we attain 
\begin{equation}\label{e:h}
\begin{split}
H(t)& =\Big(h_2(t) -\frac{R_0(0)}{2\sqrt{b}}\cA^{-1} R_-(\sqrt{b}t)\Big) u_0
- \int_0^t R_+(\sqrt{b}s)h_2'(t-s)u_0\, ds+
\\[1mm]
& \qquad
+ \cA^{-1}\int_0^t R_-(\sqrt{b}s)
\Big(\frac{1}{\sqrt{b}} h_0(t-s)-\sqrt{b} h_2 (t-s)\Big)u_0\,ds+
\\[1mm]
& \qquad
+\frac{1}{\sqrt{b}}\cA^{-1} R_-(\sqrt{b}t) u_1
+ \frac{1}{\sqrt{b}} \cA^{-1}\int_0^t R_-(\sqrt{b}s)h_1(t-s)u_1\,ds-
\\[1mm]
& \qquad
+\frac{1}{\sqrt{b}} \cA^{-1}\int_0^t R_-(\sqrt{b}s)h_2 (t-s)u_2\, ds \\[1mm]
& \qquad -\sqrt{b} \cA\int_0^t R_-(\sqrt{b}(t-s)) G\,e^{-\frac{1}{2}R_0(0)s} g(s)\,ds\,,
\end{split}
\end{equation}
which allows the understanding of the regularity of $H(t)$, along with the sought regularity properties of solutions to the MGT equation.
 
\smallskip
Notice first that in comparison with the general model equation with memory \eqref{e:memory}
the space regularity of $H(t)$ is not improved, owing to the presence of the term $h_2(t)u_0$. 
Instead, the regularity of $H_t(t)$ is improved thanks to the cancellation of the
term $R_+(\sqrt{b}t)u_0$: in fact, if $g\equiv 0$, $u_1=u_2=0$, $u_0\in X_\lambda$,
then $H_t\in C([0,T];X_\lambda)$, while $H_{tt}\in C([0,T];X_{\lambda-1})$.
However, the said cancellation (of a term depending only on $u_0$) has no effect on the remaining terms: the dependence on $u_1$ and $u_2$ is subject to the smoothing effect already described in Table~\ref{tableGENEresu} (in terms of $u_1$ and $\xi$). 
Thus, the results displayed in Table~\ref{tableMGTresu} follow. 
(The cancellation has also another significant effect: the summand $h_2(t)u_0$ decays in time, but does not propagate in space, as explained in Remark~\ref{Rem:nonpropagazione}.  
Observe that in the term
\begin{equation*}
\frac{1}{\sqrt{b}} \cA^{-1}\int_0^t R_-(\sqrt{b}(t-s))h_2 (s)u_2\, ds
\end{equation*}
one may integrate by parts, thereby obtaining
\begin{equation*}
\begin{split}
&\frac{1}{\sqrt{b}} \cA^{-1}\int_0^t R_-(\sqrt{b}(t-s))h_2 (s)u_2\, ds
= -\frac{1}{b} \cA^{-2}\Big\{h_2 (t)u_2-R_+(\sqrt{b}(t)h_2 (0)u_2
\\[1mm]
&  \myspace 
+ \int_0^t R_+(\sqrt{b}(t-s))h_2' (s)u_2\, ds\Big\}
\end{split}
\end{equation*}
that confirms the said smoothing effect.  

Using once more that the functions $h_i(t)$, $i=0,1,2$ are twice differentiable, 
it is easily seen that when $u_0\in X_\lambda$, $u_1=u_2=0$, $g=0$, then 
\begin{equation}\label{e:anomaly}
H(t)\in  C([0,T];X_\lambda)\cap C^1([0,T];X_\lambda)\cap C^2([0,T];X_{\lambda-1})\,;
\end{equation}
the regularity of $v$ and the one of $u$ follow accordingly.
In conclusion, the   representation  \eqref{e:h} of  $H(t)$ shows that all the regularity results
summarized in the rows of Table~\ref{tableGENEresu} remain valid, with the exception
of those in the first row, that are improved consistently with \eqref{e:anomaly}.
\end{proof}

\begin{remark}
\begin{rm}
We note that in particular the regularity of the mapping 
\begin{equation*}
g \longmapsto (u,u_t,u_{tt}) \qquad \text{(assuming $u_0=u_1=u_2=0$)}
\end{equation*}
for the Moore-Gibson-Thompson equation is the same as in the case of the wave equation
(as well as of the equation with memory \eqref{e:memory}). 
Hence, the last row in Table~\ref{tableMGTresu}, explicitly stated in {\em iii)} of
Theorem~\ref{t:sample}, follow readily from Theorem~\ref{t:tataru}.
\end{rm}
\end{remark}

\begin{remark}\label{Rem:nonpropagazione}
\begin{rm}
We note that $R_-(\sqrt{b}t)u_0$ and $R_+(\sqrt{b}t)u_0$ solve the wave equation, and so
the `shape' of $u_0$ is propagated in space, as in the wave equation.
Instead, the term $h_2(t)u_0$ (which decays exponentially in time) is a stationary wave
and does not propagate in the space variables.
\\
Thanks to the formulas for the solutions of the Volterra integral equations, this stationary wave appears also in the solution of the MGT equation. 
\end{rm}
\end{remark}


\section{Boundary regularity} \label{s:traces}
In this Section we establish a sharp regularity result for the normal trace
on $\Gamma=\partial \Omega$ of solutions to the the MGT equation \eqref{e:mgt},
supplemented with (homogeneous) Dirichlet boundary condition.
This result, presented as Corollary~\ref{c:traces-mgt}, follows from a boundary regularity
result pertaining to the family of wave equations with memory \eqref{e:memory}, depending
on $\xi\in L^2(\Omega)$, that is Theorem~\ref{t:traces-memory} below.
In doing so we re-obtain, in the case $\xi=0$, a result established only recently in
\cite{loreti-sforza_parma} via multiplier techniques, that is Theorem~1.1 therein.

We point out that the present approach to the analysis of wave equations with
memory enable us to pinpoint the boundary (beside the interior) regularity of solutions
in a direct and straightforward way.
The tools employed are the ones of operator and semigroup theories, along with the
view of Volterra equations; the regularity results for wave equations already available in the literature play a crucial role. 
\\
Thus, our method of proof paves the way for the derivation of appropriate boundary regularity 
results for the model equation with memory \eqref{e:memory}, as well as for the MGT equation
\eqref{e:mgt}, when supplemented with Neumann boundary conditions (a case which is known
to be drastically more difficult for the wave equation itself).
These are---to the authors' knowledge---both open problems (the former, even in the case $\xi=0$). 
 
\smallskip
Let the operator $\cT$ be the Dirichlet trace on $\Gamma = \partial \Omega$, and let 
$G=G_D$ be the Green map defined by \eqref{e:green-map} accordingly.
Then, an elementary computation which utilizes the (second) Green Theorem yields,
for $\phi\in D(A)$, the following trace result:
\begin{equation}\label{e:basic-trace-d}
G^*A\phi=-\frac{\partial \phi}{\partial\nu}\Big|_\Gamma   \qquad \forall \phi\in D(A)\,;
\end{equation}
see, e.g., \cite[Vol.~I, p.~181]{las-trig-book}.

We begin by recalling the (by now well known) result
on the boundary traces of the wave equation:

\begin{theorem}[\cite{las-lions-trig}] \label{t:traces-waves}
Let $u=u(t,x)$ be a solution to the initial/boundary value problem \eqref{e:ibvp-wave}
for the wave equation with homogeneous boundary data (i.e. $g=0$).  
Then, for every $T>0$ there exists $M=M_T$ such that
\begin{equation*}
\int_0^T\!\!\!\int_{\partial\Omega} \Big|\frac{\partial}{\partial\nu} u(x,t)\Big|^2 
d\sigma\,d t \le 
M\, \Big( \|u_0\|_{H^1_0(\Omega)}+\|u_1|_{L^2(\Omega)}^2 +\|f\|^2_{L^1(0,T;L^2(\Omega))}\Big)\,. 
\end{equation*}

\end{theorem}

\smallskip
We now see that this property is inherited by the solutions to the equation with memory
\eqref{e:memory}, and next by the solutions to the MGT equation \eqref{e:mgt}, provided a suitable compatibility condition for initial data is satisfied; see \eqref{e:compatibility} below.
The first precise statement is as follows. ({\em Cf.} \cite[Theorem~1.1]{loreti-sforza_parma}
for the case $\xi=0$.)

\begin{theorem}\label{t:traces-memory}
Under the standing Hypotheses~\ref{a:kernel-and-affine} and assuming $\xi\in L^2(\Omega)$,
let $u=u(t,x)$ be a solution to the equation with memory \eqref{e:memory},
with initial data $(u_0,u_1)$ and homogeneous boundary data.
Then, for every $T>0$ there exists $M=M_T$ such that 
\begin{equation}\label{e:traces-memory}
\int_0^T\!\!\!\int_{\partial\Omega} \Big|\frac{\partial}{\partial\nu} u(x,t)\Big|^2 
d\sigma\,d t
\le M\, \Big(\|u_0\|_{H^1_0(\Omega)}+\|u_1|_{L^2(\Omega)}^2 +\|\xi\|^2_{L^2(\Omega)}\Big)\,. 
\end{equation}
\end{theorem}

\begin{proof}
Since the equation is supplemented with Dirichlet boundary conditions, with may take
$A=\Delta$, with domain $H^2(\Omega)\cap H^1_0(\Omega)$. 
The estimate \eqref{e:traces-memory} is obtained as a simple consequence of the 
boundary regularity of solutions to the equation with memory in \eqref{ibvp-for-v},
whose convolution term was dispensed with differential terms.
Rewrite the equation in \eqref{ibvp-for-v} as
\begin{equation*}
v_{tt}=\Delta v+k_0v+\int_0^t K(t-s)v(s) ds+\cF(t)\,,
\end{equation*}
where we have set $b=1$ for the sake of simplicity (recall that $b$ must be positive),
while $\cF(t)$ is now
\begin{equation*}
\cF(t):=\big(h_2(t)\xi+h_1(t)u_1+ h_0(t) u_0\big)\,,
\end{equation*}
with the scalar functions $h_i(\cdot)$, $i=0,1,2$ introduced in \eqref{e:various-functions}.
It then follows that 
\begin{equation*}
\begin{split}
& v(t)=\underbrace{R_+(t)v_0+\cA^{-1}R_-(t)v_1}_{=: u(t)}+
\\
& \myspace + \cA^{-1}\int_0^t  R_-(t-s)\Big [k_0 v(s)+\int_0^s K(s-r) v(r) dr\Big]\,ds+
\\[1mm]
& \myspace \qquad +\,\underbrace{\cA^{-1}\int_0^t R_-(t-s) \cF(s)\,ds}_{=:T(t)}\,.
\end{split}
\end{equation*}
First of all we note that Theorem~\ref{t:traces-waves} is valid for the function $u(t)$.
Next, we observe that the integral term $T(t)$ depends on $\xi$, as the function $\cF(t)$ does.
Let us examine this dependence.
%
%
%
We recall the following version of Young inequality: given a Hilbert space $H$, if 
$h\in L^1(0,T;\mathbb{R})$ and $X\in L^2(0,T;H)$ then the convolution $X*h$ satisfies 
\begin{equation*} 
\|X*h\|_{L^2(0,T;H)}\le \|X\|_{L^2(0,T;X)}\,\|h\|_{L^1(0,T;\mathbb{R})}\,.
\end{equation*}
Then, assuming $\xi\in D(A)$, we obtain
\begin{align*}
&
\frac{\partial }{\partial \nu}  \cA^{-1}\int_0^t  R_-(t-s) h_2(s) \xi\,ds =D^*A\left [\cA^{-1}\int_0^t  R_-(t-s) h_2(s) \xi\,ds \right ]=
\\[1mm]
& \myspace =\int_0^t h_2(t-s)X(s)\, ds\,,
\end{align*}
where we set 
\begin{equation*}
X(t):= D^*A\left [\cA^{-1} R_-(t)\xi\right ]=\frac{\partial }{\partial \nu} 
\left [\cA^{-1} R_-(t)\xi\right ]\,.
\end{equation*}
Then, the (direct) inequality pertaining to wave equation establishes
\begin{equation*}
\Big\|\frac{\partial }{\partial \nu} \left [\cA^{-1} R_-(t)\xi\right ]\Big\|_{L^2(0,T;L^2(\Gamma))} \le M\|\xi\|_{L^2(\Omega)}\,,
\end{equation*}
which is extended by continuity to every $\xi\in L^2(\Omega)$.
Young inequality then gives
\begin{equation*} 
\Big \| \int_0^\cdot h_2(\cdot-s)X(s)\, ds\Big\|\le M\|\xi\|_{L^2(\Omega)}\,.
\end{equation*}
The remaining summands within $T(t)$, depending on $u_0$ and $u_1$, are continuous 
$D(A)$-valued functions, too.

Therefore, the normal trace of $v$ reads as  
\begin{align*}
&\frac{\partial}{\partial \nu} v(t)\Big|_\Gamma=-G^*A v(t)=
-G^*A\Big[u(t)+\cA^{-1}\int_0^t  R_-(t-s) \cF(s)\,ds\Big]\,ds-
\\[1mm]
& \qquad\qquad -G^*A \Big[\cA^{-1}\int_0^t R_-(t-s)\Big(k_0 v(s)+\int_0^s K(s-r) v(r) dr\Big)\,ds\Big] 
\end{align*}
and we see that there exists $M>0$ such that
\begin{align}
& \left \|-G^*A\Big[u(t)+\cA^{-1}\int_0^t  R_-(t-s) \cF(s)\,ds\Big]\,ds\right \|^2_{L^2(0,T;L^2(\Gamma))}\le
\nonumber \\[1mm]
& \myspace \le M
\left (\|u_0\|^2_{H^1_0(\Omega)}+\|u_1\|^2_{L^2(\Omega)} +\|\xi\|^2_{L^2(\Omega)}\right )\,.
\label{e:to-be-combined}
\end{align}
A similar inequality is valid for the second summand.
In fact, we know ({\em cf.} the second statement of Theorem~\ref{t:sample}) that 
$v\in C([0,T];H^1_0(\Omega))$, with continuous dependence on initial data.
Therefore, the second summand satisfies
\begin{equation*}
G^*A \Big[\cA^{-1}\int_0^t R_-(t-s)\Big(k_0 v(s)+\int_0^s K(s-r) v(r) dr\Big)\,ds\Big]
\in C(0,T;L^2(\Gamma))\,,
\end{equation*}
which combined with \eqref{e:to-be-combined} implies the sought estimate
\eqref{e:traces-memory}
\end{proof}

For the MGT equation, one obtains readily the following result.

\begin{corollary}\label{c:traces-mgt}
Let $u=u(t,x)$ be a solution to the Moore-Gibson-Thompson equation \eqref{e:mgt},
with initial data $(u_0,u_1,u_2)$ and homogeneous boundary data.
Assume $(u_0,u_1,u_2)\in H^1_0(\Omega)\times L^2(\Omega)\times H^{-1}(\Omega)$,
along with the compatibility condition
\begin{equation}\label{e:compatibility}
\xi=u_2-\Delta u_0\in L^2(\Omega)\,.
\end{equation}
Then, for every $T>0$ there exists $M=M_T$ such that
\begin{equation*} 
\int_0^T\!\!\!\int_{\partial\Omega} \Big|\frac{\partial}{\partial\nu} u(x,t)\Big|^2 
d\sigma\,d t\le 
M\, \Big( \|u_0\|_{H^1_0(\Omega)}+\|u_1\|_{L^2(\Omega)}^2 +\|u_2-\Delta u_0\|^2_{L^2(\Omega)}\Big)\,. 
\end{equation*}
\end{corollary}


\appendix
\section{Justification of Definition~\ref{d:def-solution}} \label{a:def-solutions}
Let us recall that in order to give a Definition of solutions to the MGT Equation~\eqref{e:mgt}
we proceeded as follows: formal calculations were used to reduce equation \eqref{e:mgt} to the
integro-differential equation \eqref{eq:MEMORY} and then to the Volterra integral equation
\eqref{eq:represent} in the unknown $v$. 
By definition, $u$ solves \eqref{e:mgt} when $v(t)=e^{-(R(0)/2) t} u(t)$ solves the Volterra
integral equation \eqref{eq:represent} (with $g$ replaced by $e^{-(R(0)/2) t}g(t)$). 
In this Appendix we provide a formal justification of the said Definition.
 
The argument is similar to the one used in Section~\ref{s:preliminaries-on-wave} in the
case of wave equations: we prove that the solution $u$ is smooth and can be replaced in 
both the sides of \eqref{e:mgt} 
when the initial data and the control are ``smooth'' and then we use continuous dependence as stated in Table~\ref{tableMGTresu} to justify the definition in general. 
This procedure is a bit more elaborated than the one pertaining to the wave equation,
since the third derivative (in time) comes into the picture, which requires more information
on the solutions of the wave equation.

In order to distinguish the memoryless wave equation from the equation with memory and the MGT equation, we will denote by $u_3$ the solution to equation \eqref{e:ibvp-wave} (this is because we use suitable results from \cite[\S~2.2]{PandLIBRO}, where $u_3$ solves the wave equation
when the initial data and the affine term are zero).
We assume $u_3(0) \in \cD(\Omega)$, $\frac{\partial}{\partial t} u_3\big|_{t=0}\in \cD(\Omega)$, 
$g\in \cD((0,T) \times \partial \Omega)$, where $\cD$ denotes the space of $C^\infty$ 
functions with compact support in the indicated open set (which should not be confused
with the domain of an operator), while $\partial\Omega$ is relatively open respect to itself.
The assumptions on the affine term $F(t)$ are made explicit below.
For the sake of simplicity, 
in the sequel the time derivative will be denoted by $'$.
 
\smallskip
It is known that $u_3$ is given by formula~\eqref{e:waves-explicit}: it is also clear that
if $g, f\equiv 0$, then in view of the Sobolev embedding theorems one has 
$u_3 \in C^\infty((0,T) \times \Omega)$, for every $T>0$.
Our aim is to show that a similar property holds true when $g\ne 0$, $f\ne 0 $. 
\\
Let us study separately the effects of $g$ and $f$: accordingly, we assume first $f=0$,
so that
\begin{equation*}
u_3(t)=-\cA\int_0^t R_-(s) Gg(t-s) \,ds= Gg(t)+ \int_0^t R_+(s) Gg'(t-s)\,ds\,.
\end{equation*}
As already noted we have $u_3(t)-Gg(t)\in \mathcal{D}(A)$ and 
the boundary condition is satisfied; moreover,
\begin{equation*} 
A \left(u_3(t)-Gg(t)\right)=-Gg''(t)+\int_0^t R_+(s) Gg'''(t-s)\, ds\in C^\infty\left ([0,T];L^2(\Omega)\right)\,.
\end{equation*}

Observe that, by definition,
\begin{equation*} 
A \left (u_3(t)-Gg(t)\right )=(\Delta-I) u_3(t)\in C^\infty\left([0,T];L^2(\Omega)\right)
\end{equation*}
that is $u_3(t)\in C^\infty \left([0,T];H^2(\Omega)\right)$ with suitable homogeneous boundary condition.
Analogously,
\begin{equation*} 
\cA\big\{A\big[ u_3(t)-Gg(t)\big]+Gg''(t)\Big\}=\int_0^t R_-(s)Gg^{(4)}(t-s)\, ds
\end{equation*}
which again is of class $C^\infty([0,T];L^2(\Omega))$. So we have
\begin{equation*} 
\cA\big\{A\big[ u_3(t)-Gg(t)\big]+Gg''(t)\Big\}\in C^\infty([0,T];X_1)\,, 
\end{equation*}
that is $u_3\in C^\infty\big([0,T];H^3(\Omega)\big)$.
 
By iteration we see that in the interior of $(0,T) \times\Omega$ the solution $u_3$ is of 
class $C^\infty$ and hence, when computing the derivatives, the order can be interchanged.
  
\smallskip
Let us consider now the effect of the affine term $f(t)$. 
We assume $f\in C^\infty\big([0,T]\times \Omega\big)$ and that for every fixed $t\ge 0$ 
$f(t,\cdot)\in \cD(\Omega)$, and yet possibly $f(0,\cdot)\ne 0$.

The contribution of this affine term is
\begin{equation*} 
u_3(t)=\cA^{-1}\int_0^t R_-(s)f(t-s)\, ds\in C^\infty\big([0,T]\times X_1\big)
\end{equation*}
since $f^{(n)}(0)\in \cD(A^k)$ for every couple of integers $n$ and $k$, so that
\begin{equation*}
u_3(t)\in C^\infty\big([0,T];X_k\big) \quad \text{for every $k$.}
\end{equation*}
In particular, $u_3 \in C^\infty\left( [0,T]\times\Omega\right)$ as above.

\smallskip
We now extend the obtained properties to the solutions $v$ to the Volterra integral equation
\eqref{eq:represent} so that it is possible to track back the computation and to see that equality~\eqref{e:mgt} holds pointwise (when the boundary control and the initial conditions have the stated regularity, $u_2\in \cD(\Omega)$ included).

We confine ourselves to examine the effect of the boundary data $g$ (the effect of initial
data can be examined in a similar way). Moreover, multiplication by $e^{-R(0)t/2}$ does not
affect the desidered results and hence is ignored; i.e. we assume $v(t)\equiv u(t)$.

Because equation \eqref{eq:represent} has the form of equation (2.25) in \cite[\S~2]{PandLIBRO} (the notations are easily adapted, in particular $b$ is substituted by $c^2$ in \cite{PandLIBRO})
following the proof of \cite[Theorem~2.4, item~2]{PandLIBRO} we see that 
$y(t)=v(t)-Gg(t)$ solves 
\begin{equation*}
y(t)=\left(u_3(t)-Gg(t)\right)+\int_0^t L(s)Gg(t-s)\,ds+\int_0^t L(s)y(t-s)\,ds 
\end{equation*}
so that 
\begin{equation*}
\cA y(t)=\int_0^t \cA L(s) Gg(t-s)\, ds+\int_0^t L(s)\cA y(t-s)\, ds
\end{equation*}
(note that $\cA L(t)$ is a continuous operator for every $t$). 
It then follows that $y(t)\in C^\infty\left([0,T];X_1\right)$.

Exploiting the definition of $L(t)$ and integrating by parts the integral which contains 
$g(t)$ we see that $y(t)\in C^\infty\left ([0,T];X_2\right)$. 
Iterating this procedure, we obtain $u\in C^\infty\left([0,T]\times\Omega\right)$.
Using this regularity result we can track back the computation leading to the fact that $u(t)$ solves the MGT equation, including the fact that the Laplacian and the time derivative can be interchanged.

\medskip
\section*{Acknowledgements}
\noindent
The research of F.B. was partially supported by the Universit\`a degli Studi di Firenze
under the Project {\em Analisi e controllo di sistemi di Equazioni a Derivate Parziali di evoluzione}, and by the GDRE (Groupement de Recherche Europ\'een) ConEDP ({\em Control of PDEs}).
F.B. is a member of the Gruppo Nazionale per l'Analisi Mate\-ma\-tica, la Probabilit\`a 
e le loro Applicazioni (GNAMPA) of the Istituto Nazionale di Alta Matematica (INdAM). 
\\
The research of L.P. was partially supported by the Politecnico di Torino, and by the GDRE (Groupement de Recherche Europ\'een) ConEDP ({\em Control of PDEs}).
L.P. is a member of the Gruppo Nazionale per l'Analisi Matematica, la Probabilit\`a 
e le loro Applicazioni (GNAMPA) of the Istituto Nazionale di Alta Matematica (INdAM).


\medskip
\begin{thebibliography}{99}
 
\bibitem{belleni}
{\sc A.~Belleni-Morante},
An integro-differential equation arising from the theory of heat conduction in rigid materials
with memory,
\emph{Boll.~Un.~Mat.~Ital.} {\bf 5} 15-B (1978), 470-482. 

\bibitem{corduneanu}
{\sc C.~Corduneanu},
{\em Integral Equations and Applications},
Cambridge University Press, 2010.

\bibitem{delloro-pata_2016}
{\sc F.~Dell'Oro and V.~Pata},
On the Moore-Gibson-Thompson equation and its relation to viscoelasticity,
{\em Appl.~Math.~Optim.} {\bf 76} (2017), 641-655.

\bibitem{fattorini}
{\sc H.O.~Fattorini},
{\em Second order linear differential equations in Banach spaces},
North-Holland Publishing Co., Amsterdam, 1985.

\bibitem{Grisvard} {\sc P. Grisvard}, 
Controlabilit\'e exacte des solutions de l'\'equation des ondes en pr\'esence de singularit\'es,
{\em J. Math. Pures Appl.} {\bf 68} (1989), 215-259.

\bibitem{jordan_2009}
{\sc P.~Jordan}, 
Nonlinear acoustic phenomena in viscous thermally relaxing fluids: Shock bifurcation
and the emergence of diffusive solitons,
{\em The Journal of the Acoustical Society of America} {\bf 124} (2009), no.~4, 
2491-2491.

\bibitem{kalt-etal_2011}
{\sc B.~Kaltenbacher, I.~Lasiecka and R.~Marchand},
Wellposedness and exponential decay rates for the Moore-Gibson-Thompson equation arising
in high intensity ultrasound, 
{\em Control Cybernet.} {\bf 40} (2011), no.~4, 971-988. 

\bibitem{kalt_2015}
{\sc B.~Kaltenbacher}, 
Mathematics of nonlinear acoustics,
{\em Evolution Equations Control Tehory} {\bf 4} (2015), no.~4, 447-491.

\bibitem{kalt-las-posp_2012}
{\sc B.~Kaltenbacher, I.~Lasiecka and M.~Pospieszalska}, 
Wellposedness and exponential decay of the energy in the nonlinear 
Jordan-Moore-Gibson-Thompson equation arising in high intensity ultrasound, 
{\bf M3AS} 22 (2012), 1250035, 34~pp.

\bibitem{las-jee_2017}
{\sc I.~Lasiecka},
Global solvability of Moore-Gibson-Thompson equation with memory arising in nonlinear
acoustics,
{\em J.~Evol.~Equ.} {\bf 17} (2017), no.~1, 411-441.

\bibitem{las-trig-cos}
{\sc I.~Lasiecka and R.~Triggiani},
A cosine operator approach to $L_2(0,T;L_2(\Gamma))$-boundary input hyperbolic equations,
\emph{Appl. Math. Optim.}~\textbf{7} (1981), 35-93.

\bibitem{las-lions-trig}
{\sc I.~Lasiecka, J.-L.~Lions and R.~Triggiani}, 
Nonhomogeneous boundary value problems for second order hyperbolic operators,
{\em J. Math. Pures Appl.} (9) {\bf 65} (1986), no.~2, 149-192. 

\bibitem{las-trig_wave1}
{\sc I.~Lasiecka and R.~Triggiani},
Sharp regularity theory for second order hyperbolic equations of Neumann type, I.
$L^2$ nonhomogeneous data,
{\em Ann. Mat. Pura Appl.} {\bf 157} (1990), no.~4, 285-367. 

\bibitem{las-trig_wave2}
{\sc I.~Lasiecka and R.~Triggiani},
Regularity theory of hyperbolic equations with nonhomogeneous Neumann boundary conditions, 
II. General boundary data,
{\em J. Differential Equations} {\bf 94} (1991), no.~1, 112-164. 

\bibitem{las-trig-book}
{\sc I.~Lasiecka and R.~Triggiani},
{\em Control theory for partial differential equations: continuous and approximation theories,
I. Abstract parabolic systems; II. Abstract hyperbolic-like systems over a finite time horizon}, 
Encyclopedia of Mathematics and its Applications, {\bf 74},
Cambridge University Press, Cambridge, 2000. xxii+644+I4 pp.

\bibitem{lions_1988}
{\sc J.-L.~Lions}, 
{\em Contr\^olabilit\'e exacte, perturbations et stabilisation de syst\`emes distribu\'es},
Tome 1 [Exact controllability, perturbations and stabilization of distributed systems. Vol. 1] with appendices by E.~Zuazua, C.~Bardos, G.~Lebeau and J.~Rauch,
{\em Recherches en Math\'ematiques Appliqu\'ees} [Research in Applied Mathematics],
Vol.~8, Masson, Paris, 1988. x+541 pp.

\bibitem{loreti-sforza_parma}
{\sc P.~Loreti and D.~Sforza},
Hidden regularity for wave equations with memory, 
{\em Riv.~Mat.~Univ. Parma} {\bf 7} (2016), no.~2, pp.~391-405. 

\bibitem{maccamy_1977}
{\sc R.C.~MacCamy}, 
A model for one-dimensional, nonlinear viscoelasticity,
{\em Quart. Appl. Math.} {\bf 35} (1977), 21-33.

\bibitem{marchand-etal_2012}
{\sc R.~Marchand, T.~McDevitt and R.~Triggiani},
An abstract semigroup approach to the third-order Moore-Gibson-Thompson partial differential
equation arising in high-intensity ultrasound: structural decomposition, spectral analysis,
exponential stability, 
{\em Math. Methods Appl. Sci.} {\bf 35} (2012), no.~15, 1896-1929. 

\bibitem{moore-gibson_1960}
{\sc F.K.~Moore and W.E.~Gibson},
Propagation of weak disturbances in a gas subject to relaxation effects,
{\em Journal of Aerospace Sciences and Technologies} {\bf 27} (1960), 117-127.

\bibitem{PandAMO}
{\sc L.~Pandolfi},
The controllability of the Gurtin-Pipkin equation: a cosine operator approach,
\emph{Appl.~Math.~Optim.}~\textbf{52} (2005), 143-165;
(a correction in \emph{Appl. Math. Optim.}~\textbf{64} (2011), 467-468)

\bibitem{PandLIBRO}
{\sc L.~Pandolfi},
{\em Distributed systems with persistent memory. Control and moment problems}, 
SpringerBriefs in Electrical and Computer Engineering. 
SpringerBriefs in Control, Automation and Robotics, Springer, Cham, 2014. x+152 pp.

\bibitem{pruess_1993}
{\sc J.~Pr\"uss},
{\em Evolutionary Integral Equations and Applications},
Monographs in Mathematics Vol.~ 87, Birkh\"auser Verlag, Basel, 1993;
also: [2012] reprint of the 1993 edition, Modern Birkh\"auser Classics, 
Birkh\"user/Springer Basel AG, Basel, 1993. xxvi+366 pp.

\bibitem{skubacevskii} 
{\sc A.L.~Skubacevski\v{i}},
\emph{Elliptic functional-differential equations and applications.}  
Birkh\"auser Verlag, Basel, 1997.

\bibitem{sova_1966}
{\sc M.~Sova},
Cosine operator functions, {\em Rozprawy Mat.} {\bf 49} (1966), 47 pp. 
 
\bibitem{tataru_1998}
{\sc D.~Tataru},
On the regularity of boundary traces for the wave equation.
{\em Ann.~Scuola~Norm.~Sup. Pisa} Cl.~Sci. (4) {\bf 26} (1998), no.~1, 185-206.  

\bibitem{thompson_1972}
{\sc P.A.~Thompson}, 
{\em Compressible-fluid Dynamics}, McGraw-Hill, New York, 1972.

\bibitem{triggiani_2016}
{\sc R.~Triggiani},
Sharp regularity theory of second order hyperbolic equations with Neumann boundary
control non-smooth in space,
{\em Evol.~Equ. Control Theory} {\bf 5} (2016), no.~4, 489-514. 

\end{thebibliography}
\end{document}